\documentclass[12pt,leqno]{amsart}
\usepackage{amsmath,amssymb,amscd,yfonts,fullpage}
\usepackage{enumerate}
\usepackage{color}
\numberwithin{equation}{section}
\newtheorem{lemma}{Lemma}[section]
\newtheorem{theorem}{Theorem}[section]

\newtheorem{proposition}{Proposition}[section]
\newcommand{\bp}{\begin{proposition}}
\newcommand{\ep}{\end{proposition}}
\newcommand{\bth}{\begin{theorem}}
\newcommand{\et}{\end{theorem}}
\newcommand{\be}{\begin{equation}}
\newcommand{\ee}{\end{equation}}
\newcommand{\bal}{\begin{align}}
\newcommand{\eal}{\end{align}}
\newcommand{\bi}{\begin{itemize}}
\newcommand{\ei}{\end{itemize}}
\newcommand{\la}{\label}
\newcommand{\noi}{\noindent}

\newcommand{\boE}{{\bold E}}

\renewcommand{\a}{\alpha}
\renewcommand{\b}{\beta}
\renewcommand{\d}{\delta}
\newcommand{\D}{\Delta}
\newcommand{\e}{\varepsilon}
\newcommand{\g}{\gamma}
\newcommand{\G}{\Gamma}
\renewcommand{\l}{\lambda}
\renewcommand{\L}{\Lambda}

\newcommand{\var}{\varphi}
\newcommand{\s}{\sigma}

\renewcommand{\th}{\theta}
\renewcommand{\O}{\Omega}
\renewcommand{\o}{\omega}
\newcommand{\z}{\zeta}

\newcommand{\x}{\times}
\renewcommand{\i}{\infty}
\newcommand{\p}{\partial}

\newcommand{\bE}{{\mathbb E}}

\newcommand{\bP}{{\mathbb P}}
\newcommand{\bR}{{\mathbb R}}

\newcommand{\bQ}{{\mathbb Q}}
\newcommand{\bZ}{{\mathbb Z}}

\newcommand{\bT}{{\mathbb T}}

\newcommand{\bs}{{\bigskip}}
\newcommand{\ms}{{\medskip}}

\newcommand{\cD}{{\mathcal D}}

\newcommand{\cG}{{\mathcal G}}
\newcommand{\cL}{{\mathcal L}}

\newcommand{\cA}{{\mathcal A}}
\newcommand{\cM}{{\mathcal M}}
\newcommand{\cN}{{\mathcal N}}
\newcommand{\cP}{{\mathcal P}}
\newcommand{\cF}{{\mathcal F}}
\newcommand{\cH}{{\mathcal H}}
\newcommand{\cB}{{\mathcal B}}

\newcommand{\cI}{{\mathcal I}}
\newcommand{\cC}{{\mathcal C}}

\newcommand{\cS}{{\mathcal S}}
\newcommand{\cT}{{\mathcal T}}
\newcommand{\cJ}{{\mathcal J}}
\newcommand{\cZ}{{\mathcal Z}}

\newcommand{\cQ}{{\mathcal Q}}
\renewcommand{\leq}{\leqslant}
\renewcommand{\geq}{\geqslant}

\newenvironment{example}{\refstepcounter{theorem}\par\medskip\noindent{\bf
Example~\thetheorem~~}}{\unskip\nobreak\hfill\hbox{}}

\newenvironment{remark}{\refstepcounter{theorem}\par\medskip\noindent{\bf
Remark~\thetheorem~~}}{\unskip\nobreak\hfill\hbox{}}

\newenvironment{definition}{\refstepcounter{theorem}\par\medskip\noindent{\bf
Definition~\thetheorem.}}{\unskip\nobreak\hfill\hbox{\medskip \smallskip}}

\title{The random Arnold Conjecture: a new probabilistic Conley-Zehnder Theory for symplectic maps}

\author{\'Alvaro Pelayo\,\,\,\,\,\,\,\, Fraydoun Rezakhanlou}

\date{}

\begin{document}

\maketitle

\begin{abstract}
We take the first steps to develop Conley\--Zehnder Theory, as conjectured by Arnold, in the world
of probability. As far as we know, this paper provides the first probabilistic theorems about the density of fixed points of symplectic twist
maps in dimensions greater than $2$. In particular we will show that, when the analogue conditions to classical Conley\--Zehnder 
theory hold,  quasiperiodic symplectic twist maps have infinitely many fixed points almost surely. The paper contains also
a number of theorems which go well beyond the quasiperiodic case.
\end{abstract}

\section{Introduction}

 Conley\--Zehnder theory, as conjectured by Arnold, is one of the the great achievements at the 
intersection of symplectic geometry and Hamiltonian dynamics in the past few decades.

 The motivation for these works can be traced back to Poincar\'e and later to the first developments in
symplectic topology.

In the present paper we take the first steps to develop Conley\--Zehnder Theory in a probabilistic setting. 
 One of our main theorems says that a quasiperiodic twist symplectic map has 
infinitely many fixed points almost surely, provided the analogue conditions to those imposed by Conley and Zehnder in 
their famous theorem for tori, hold. 

A main tool of the paper is the Ergodic Theorem, which allows us to control
the behavior of random symplectic maps in analogy with how topological assumptions such as
compactness are used in Conley\--Zehnder theory. We also use the Ergodic Theorem to evaluate
the density of fixed points.

 \subsection{Poincar\'e's theorem on area preserving maps: from classical to random}
 
The work of Henri Poincar\'e in classical mechanics~\cite{Po93} led him to the famous Poincar\'e\--Birkhoff 
Theorem \cite{Po12} concerning fixed points of area preserving twist maps of an annulus, which he stated in 1912.  It was
Birkhoff~\cite{Bi1,Bi2} who finally proved the result in 1925. 

The result essentially says that such a map
always has at least two fixed points, and these points are genuinely different.  

This result motivated us to pursue similar statements in the context of probability theory, and we
took the first steps to do this in our initial paper in the subject~\cite{PR} published in 2018 (see Section~\ref{sec2} for a
very brief account of the main results of this paper), where we proved
that there is a statement, similar to the  Poincar\'e\--Birkhoff Theorem, for area preserving maps which
are random. 

In \cite{PR} we rely heavily on ``finite dimensional" methods, notably the theory of
generating functions, which allows one to reduce infinite dimensional proofs to the finite dimensional
case.  More concretely, we used the ideas of Chaperon~\cite{Ch1984,Ch1984b, Ch1989} and Viterbo~\cite{Vi2011}.

There is an essential mathematical difference between the classical and the random versions of this result,
and which is also at the heart of the proofs: while the classical result is more topological because it 
is established in setting of compactness, the random result is more analytic because it only makes
sense in the non\--compact world. Hence, in the proof we gave, we used mostly analysis.

 From the
point of view of what statements to expect, in the random setting one expects to have \emph{infinitely
many fixed points}, and to come in families which are genuinely different too.

\subsection{Moving to higher dimensions: the Arnold Conjecture}

From the point of view of symplectic geometry, a result in dimension $2$, while interesting is not
entirely satisfactory. It was Arnold who had the great insight of formulating an analogue of the result
by Poincar\'e\--Birkhoff in higher dimensions. He saw, that one should consider ``symplectic maps"
instead of ``volume preserving maps", and formulated the famous Arnold Conjecture.  
This conjecture has generated an immense amount of research in symplectic geometry in the past few decades. 

Essentially the conjecture says that if $(M,\omega)$ is a compact symplectic manifold, then
 any time periodic Hamiltonian diffeomorphism has at least as many fixed points as a smooth function has critical points.

\subsection{A breakthrough by Conley and Zehnder, and beyond}

Conley and Zehnder~\cite{CoZe1983} made the first breakthrough on the
Arnold conjecture, proving it for the $2d$\--dimensional torus; more
precisely they proved that  any smooth symplectic map of the $2d$\--dimensional torus that is isotopic to identity has at least 
$2d+1$ many fixed points. 

The work by Conley and Zehnder was followed by major works by Floer where he developed the ideas of what now is
known as Floer theory~\cite{Fl88,Fl89,Fl89b,Fl91}, and works of many others,
including Hofer, Hofer-Salamon, Liu-Tian,  Ono, and Weinstein~\cite{Ho85,HS95,LT98,On95,We86}.

 \subsection{Goal of this paper: towards a probabilistic Arnold Conjecture and a proof of the random Conley\--Zehnder Theorem}
 
Our goal in this paper is to take the first steps to understand the Arnold Conjecture in the world of probability, by
establishing the Conley\--Zehnder theorem for random symplectic twist maps. 

At this time much of  the technical machinery that is needed to remove the assumption ``twist" is not yet available, nonetheless we believe that a much more general result will be
possible in the future, and as such we state it as  a meta-goal in Section~\ref{sec1}.

The statements and proofs we present in the paper we believe are quite new, in the sense that we are not aware of results in this direction, beyond what
we did in dimension $2$ in our paper from five years ago~\cite{PR}. Indeed ~\cite{PR}  concerns area preserving maps in dimension $2$, so it is not in
that sense a very ``symplectic" paper, while in the current paper we treat any dimension, in the spirit of the original Arnold Conjecture and Conley\--Zehnder theory;
so the context of our current paper is indeed, symplectic.

  \subsection{Novelty of the paper: statements and proof techniques}
   
 Our  point of view in our previous paper~\cite{PR} in dimension $2$ was mainly the classical theory of generating functions, 
 more specifically Chaperon's viewpoint \cite{Ch1984,Ch1984b, Ch1989}.  This point of view has been further advocated by Viterbo~\cite{Vi2011}.
 
We believe that the ideas of the present paper --- both involved in the statements and in the proofs ---
 are new, and are developed from combining ergodic and symplectic  methods. Indeed,  as far as we know, our paper 
provides the first probabilistic  theorems about the density of fixed points of symplectic twist maps in dimensions greater than $2$.

We recommend Hofer\--Zehnder~\cite{HoZe1994} and Polterovich~\cite{Pol01} for treatments of different aspects of symplectic topology.   
We refer to  Gol\'e~\cite{Go2001} for a treatment of symplectic twist maps and to Adler\--Taylor~\cite{AT07,AT09} for treatments of certain
 geometric aspects of   randomness. 
 
  In particular we refer to~\cite{AT07,AT09,AW} for  thorough discussions of Kac-Rice type formulas for level sets of Gaussian random fields.

\subsection{Structure and main achievements of the paper} 

In Section~\ref{sec1} we formulate the main goal of the paper (Meta-Goal and Stochastic Analogue of Conley\--Zehnder Theorem) and formulate two of our main results: 
Theorem~\ref{MGT} and  Theorem~\ref{th1.3}.  These two results concern quasiperiodic symplectic maps and are simpler to state, but the paper goes well beyond this case, so in this section we also announce the other main results of the paper: Theorem~\ref{th3.3},  Proposition~\ref{prop4.1}, Theorem~\ref{th4.2} and Theorem~\ref{th5.1}.
 
In Section~\ref{pre} we study the existence of generating functions.  Notably, in Theorem~\ref{th3.3} we prove that the ``generating function" is stationary.  This allows us in
Proposition~\ref{prop4.1} to have an almost sure candidate for the density of fixed points, meaning by ``density" the number of fixed points in a box of side 
$2\ell$ divided
by the volume of the box.

 This poses the problem of deciding whether the density is positive, which we achieve in Section~\ref{cpsp} and Section~\ref{qsp} by deriving an explicit formula
for this density in two cases: Theorem~\ref{th4.2} and Theorem~\ref{th5.1}. One of these cases has to do with quasiperidic symplectic twist diffeomorphisms which would lead to Theorem~\ref{MGT} and Theorem~\ref{th1.3}. 
We would prefer to give a rather informal statements of these theorems in this section, and provide detailed and precise 
versions of these theorems later in the paper as
 Theorems~\ref{th5.1} and \ref{th6.2}.

Section~\ref{ode} is devoted to the properties of time-one map of stationary Hamiltonian ODEs. These properties would allow us to deduce Theorem~\ref{th1.3} (equivalently Theorem~\ref{th6.2}) from Theorem~\ref{MGT} (equivalently Theorem~\ref{th5.1}).

Finally, in the appendix (Section~\ref{sec2}) we review the $2$\--dimensional case, that is, the random Poincar\'e\--Birkhoff Theorem. This case is much simpler to 
discuss and relies on more standard tools, so we believe that it  can serve as a warm up for the results that  of the present paper.

\subsubsection*{Acknowledgments}
Both authors have been generously funded by a BBVA (Bank Bilbao Vizcaya Argentaria) Foundation Grant for Scientific Research Projects with project title \emph{From Integrability to Randomness in Symplectic and Quantum Geometry (FITRISAQG)}. 

During the preparation for some parts of this work, FR was hosted by Max Fathi at Universit\'e Paris Cit\'e during June of 2023. He is thankful for 
the invitation and hospitality of UP Cit\'e. During the final stages of this work
(July 2023), AP was visiting the University of Cantabria and UIMP
(Universidad International Men\'endez Pelayo), and he is thankful to these
institutions for their hospitality.

 AP is also thankful to the Dean of the School of Mathematics of the Complutense University of Madrid Antonio Br\'u and the Chair of the Department of Algebra, Geometry and Topology Rutwig Campoamor for their support and  excellent resources he is being provided with to carry out the aforementioned BBVA project.

\section{Meta-Goal, main results on quasiperiodic symplectic twists, and announcements of results beyond the quasiperiodic case} \label{sec1} 

We very briefly review the classical Conley\--Zehnder theory and then establish its probabilistic analogue, which is our main result.

 \subsection{Conley\--Zehnder Theory}

We are interested in extending to the stochastic setting the following classical result of Conley and Zehnder.

 \begin{theorem}[Conley\--Zehnder~\cite{CoZe1983}] \label{1.1} 
Every smooth time $1$\--periodic Hamiltonian vector field on the standard torus $\bT^{2d}$ has at least $2d+1$ contractible periodic orbits.
 \end{theorem}

 Alternatively, if we write $\var$ for the time one map of the flow of such Hamiltonian vector field, then $\var$ is a symplectic map that has at least $2n+1$ fixed points. Writing 
 $\Phi:\bR^{2d}\to\bR^{2d}$
 for the lift of $\var$, we have a symplectic map  that 
 has at least $2d+1$ fixed points in any box of side length $1$.
 We may state Conley\---Zehnder Theorem in terms of $\Phi$.
 
 \begin{theorem}[Conley\--Zehnder~\cite{CoZe1983}]  \label{cz2}
 Let $\Phi:\bR^{2d}\to\bR^{2d}$ be a symplectic
 diffeomorphism that is homologous to identity. Additionally assume  that $$\o(x):=\Phi(x)-x$$ is $1$-periodic, and
 \be\la{eq1.1}
 \int_{[0,1]^{2d}}\, \o(x)\, {\rm d}x=0. \nonumber
 \ee
  Then $\Phi$ has at least $2d+1$ fixed points in the set
 $[0,1)^{2d}$.
\end{theorem}

Chaperon \cite{Ch1984} carried out a proof of the Conley\--Zehnder Theorem using generating functions, and the present paper pushes these ideas
further.

 \subsection{Random Conley\--Zehnder Theory: quasiperiodic case (the simplest beyond periodic)}

For the stochastic analogue of \cite{CoZe1983}, we take a 
symplectic diffeomorphism $\Phi:\bR^{2d}\to\bR^{2d}$ of the form 
$$
\Phi(x)=x+\o(x), 
$$
where  $\o$ is selected randomly, and the map
$$x\mapsto\o(x)$$ is a stationary 
process with respect to the $2d$-dimensional translation
$$\th_a\o(x)=\o(x+a).$$
 
 Our typical result would assert that 
even if $\o$ is not periodic, then generically the corresponding $\Phi$ would have infinitely 
many fixed points. 

In fact we will use probabilistic means to select a generic $\Phi$. To explain this, let us set
 $$
  \cS:=\Big\{\Phi \colon \bR^{2d}\to\bR^{2d} \,:\, \Phi\, \textup{ symplectic diffeomorphism},\, \o:=\Phi-{\rm id} \,   \textup{ bounded} \Big\},
   $$
and define $$\th'_a\Phi(x)=\Phi(x+a)-a$$  so that $$\th'_a\Phi(x)-x=\th_a\o(x).$$  
We equip $\cS$ with the topology of ${\rm C}^1$ norm and consider the $\s$-algebra $\cB$ of Borel subsets of $\cS$.  
 \begin{definition}
 We say  that a probability measure $\cP$ on $\cS$ is \emph{$\th'$-invariant and 
 ergodic} if the following conditions are true:
 \begin{itemize}
 \item[(i)] For every $A\in\cB$ we have that $\cP\big(\th'_a A\big)=\cP(A)$.
 \item[(ii)] If there exists $A\in\cB$ such that $\th'_aA=A$ for all
 $a\in\bR^{2d}$, then $\cP(A)\in\{0,1\}$.
 \end{itemize}
In the same vein, we can talk about a probability measure $\bP$ that is 
 \emph{$\th$-invariant and  ergodic}
 \end{definition}

\begin{example}{\em {(Almost periodic\--twists)}}
To ease the notation, we write $n$ for $2d$.
Given a continuous function $\bar \Phi:\bR^{n}\to\bR^{n}$, let us assume that the set $\G=\{\th'_a\bar\Phi\}$ is precompact with respect to the topology of uniform convergence. 
We write $\G'$ for the topological closure of $\G$.
 By the classical theory of almost periodic functions,  the set $\G'$ can be turned into a compact topological group and for $\cP$, we may choose a normalized {\em{Haar}} measure on $\G'$. We say $\bar\Phi$ is \emph{quasiperiodic} if the group
$\G'$ is finite dimensional (and therefore isomorphic to a torus). 
More concretely,
let $\bar \o:\bR^N\to\bR$ be a $1$-periodic ${\rm C}^2$ function,
and let $A\in\bR^{N\x n} $ be a matrix. 
Then the map 
$$\bar\Phi(x)=x+\bar\o({Ax}),$$ is quasiperiodic.
\end{example}

\bs\noi
{\bf Meta-Goal (Stochastic Analogue of Conley\--Zehnder Theorem):}       
    \emph{Let $\cP$ be a  $\th'$-invariant ergodic  measure  on $\cS$ such that
  \be
  \la{eq1.2}
  \int_{\cS} \Phi(0)\ \cP({\rm d}\Phi)=0. \nonumber
  \ee
   Assume also that $\Phi$ is homologous to the identity map, 
$\cP$-almost surely. Then  $\Phi$ has infinitely many fixed point 
 $\cP$-almost surely.}
 
 \vspace{2mm}

We establish this Meta-Goal in an important case:
We say that  $$\Phi(q,p)=\big(Q(q,p),P(q,p)\big)$$ is {\em twist} if for every $p\in\bR^d$, the map $q\mapsto Q(q,p)$ is a diffeomorphism
of $\bR^d$.

\begin{theorem} \label{MGT}
The Meta-Goal holds when $\Phi$ is a quasiperiodic symplectic twist. 
\end{theorem}

We refer to section~5.1 for the regularity of $\Phi$ and a more detailed 
statement of Theorem~\ref{MGT}.

\ms
One natural way of producing such a stochastic symplectic map is by using time one map of a Hamiltonian ODE for which the Hamiltonian function is a stationary process with respect to the translation.

 To prepare for the statement of our results, let us write $\cH_0$ for the set of  ${\rm C}^2$ functions
 $$
 H:\bR^{2d}\x \bR\to\bR,
 $$
 such that:
 \begin{itemize}
 \item[(i)]
  $H(x,t+1)=H(x,t)$ for all $(x,t)\in\bR^{2d}\x \bR$,
  \item[(ii)]
  $\nabla H$ is uniformly bounded. 
  \end{itemize}
  We also define translation of $H\in\cH_0$ by
 \[
 \th_aH(x,t)=H(x+a,t).
 \]
 We equip $\cH_0$ with the topology of ${\rm C}^2$ norm and consider the $\s$-algebra $\cB$ of Borel subsets of $\cH_0$.

 We write $$X_H=J\nabla H(x,t)$$ for the Hamiltonian vector field associated with the Hamiltonian function $H$, and $\phi^H_t$
 for the flow of $X_H$.  The following is a stochastic analogue  of Theorem~\ref{1.1}. The regularity of $H$ in the statement
 will be made precise later in the paper (Theorem~\ref{th6.2}).
 
 \begin{theorem} \label{th1.3} 
 Let $\bP$ be a $\th$-invariant and 
 ergodic probability measure on $\cB$. If $H$
is sufficiently small, then the Hamiltonian vector field $X_H$ has infinitely many $1$-periodic orbits $\bP$-almost 
 surely.
 \end{theorem}

 \subsection{Random Conley\--Zehnder Theory: well\--beyond the quasiperiodic case}

It is important to note that our results go well beyond the quasiperidic case covered in the theorems presented in this section, as one can see from
  the theorems proven in Sections~\ref{pre}\--\ref{qsp}.

Indeed, the main result of Section~\ref{pre}, and probably the hardest and most substantial
 result of this paper, is Theorem~\ref{th3.3} which reduces the problem of
counting fixed points to counting the critical points of a stationary process.

In two cases we provide an explicit formula for the density of fixed points, which is achieved in Section~\ref{cpsp} (Theorem~\ref{th4.2}) in one case (in the case
that some random variable has a density) and then the second case of  Section~\ref{qsp} is the quasiperiodic case (Theorem~\ref{th5.1}); this is the only section where
quasiperiodic case appears at all, but we stated results earlier for this case because it is simpler to formulate.

\subsection{Examples of stationary Hamiltonian functions} 
It is important to note that the quasiperiodic case is {\bf the least random of all} and in that sense the least interesting
 from the point of view of stochastic processes. It is also the simplest case to deal with. We now describe some examples of stationary probability measures
on $\cH_0$. We equip $\cH_0$ with the topology of uniform convergence.

\begin{example}\la{ex2.7} {\em{(Periodic Hamiltonian Functions)}}
As the simplest example, take any $H_0(x,t)$ in $\cH_0$, that is $1$-periodic in $x$-variable, and set
\begin{equation}\label{eq2.1}
O(H_0)=\left\{\th_aH_0\ :\ a\in\bR^{2d}\right\}.
\end{equation}
Since $H_0$ is a $1$-periodic function, 
the set $O(H_0)$ is homeomorphic to $\bT^{2d}$. Under this homeomorphism, the translation $\th$ becomes the standard translation $\Theta$ on $\bT^{2d}$.
The only $\th$-invariant probability measure $\bP$ on $O(H_0)$ is the push forward of the Lebesgue measure on $\bT^{2d}$ under the map
$a\mapsto \th_a H_,$. 
\end{example}

\begin{example}\la{ex2.8} {\em{(Quasiperiodic Hamiltonian Functions)}}
Given $N\geq n$, pick a ${\rm C}^2$ function 
$K_0(\o,t)$, 
$$ K_0:\bT^N\times \bR\to\bR,$$
 that is $1$-periodic in $t$.
Pick a matrix $A\in\bR^{N\x n}$ that satisfies the following condition:
 \begin{equation}\label{eq2.2}
mA=0,\ \ \ m \in \mathbb{Z}^N\ \ \Rightarrow\ \  m=0.  
\end{equation}
Let
$$H_0(x,t)=K_0(Ax,t),$$ 
and define $O(H_0)$ as in \eqref{eq2.1}. Note that if $N>2d$, the set 
$O(H_0)$ is not closed. However, the condition \eqref{eq2.2} 
guarantees that its topological closure $\overline{O(H_0)}$ consists of
functions of the form 
\[
H(x,t,\o):=K_0(\o+Ax,t),
\]
with $\o\in\bT^N$.
(Here we regard $\bT^N$ as $[0,1]^N$ with $0=1$, and $\o+Ax$ is a {{\rm Mod}}\ $1$ 
summation.)
Assume that $\bP$ is concentrated on the set $\overline{O(H_0)}$.
Again, since $\bP$ is $\th$-invariant, the pull-back of $\bP$ with respect to the transformation
$\o\in\bT^N\mapsto H(\cdot,\cdot,\o)$ can only be the 
Lebesgue measure on $\bT^N$. Note that our main result Theorem~\ref{th1.3} only guarantees
the existence of $1$-periodic orbits for $H(\cdot,\cdot,\o)$, for $\bP$-almost all choices of $\o$.
\end{example}

\begin{example}\la{ex2.9}{\em {(Almost periodic Hamiltonian Functions)}}
Given a function $H_0\in\cH$, let us assume that the corresponding 
$O(H_0)$ is precompact with respect to the topology of uniform convergence. 
 By the classical theory of almost periodic functions,  the set $\overline{O(H_0)}$ can be turned to a compact topological group and for $\bP$, we may choose a normalized {\em{Haar}} measure on $\overline{\O(H_0)}$.
\end{example}

\begin{example}\la{ex2.10}{\em{(Lorenz gas type models)}}
Let us write $\O_0$ for the set of discrete subsets of $\bR^{d}$.
We also write $\bQ_0$ for the law of a Poisson point process of intensity one 
on $\O_0$. We set $\O=\O_0\x\bT^{d}$, and $\bQ=\bQ_0\x \l$,
where $\l$ denotes the Lebesgue measure of $\bT^{d}$. 
On $\O_0$, we have a
natural translation that is denoted by $\tau$: For $\o_0=\{q_i:\ i\in I\}$,
we define
\[
 \tau_q\o_0=\{q_i-q:\ i\in I\}.
\]
As before, let us write $\Theta$ for the translations on the torus $\bT^{d}$.
We define a translation $\hat\th$ on $\O$ by
\[
\hat\th_{(q,p)}(\o_0,a)=\big(\tau_q \o_0,\Theta_p a\big).
\]
The measure 
$\bQ$ is $\hat\th$ invariant and ergodic. Pick a ${\rm C}^2$ function
$K^0(p,t)$ that is $1$-periodic in all the coordinates of $(p,t)$, and a ${\rm C}^2$ function
$V(q,t)$ that is of 
compact support in $q$ and $1$-periodic in $t$.
 Given a realization of $\o=(\o_0,a)$, 
with $\o_0=\{q_i:\ i\in I\}$, we define
a Hamiltonian function
\[
H(q,p,t,\o):=K^0(p+a,t)+\sum_{i\in I}V(q-q_i,t).
\]
The map $\o \mapsto H(\cdot,\cdot,\cdot,\o)$ pushes forward the probability measure
$\bQ$ to a probability measure $\bP$ that is $\th$-invariant and ergodic.
Note that if $x(t)=(q(t),p(t))$ solves the corresponding Hamiltonian ODE
$\dot x=J\nabla H(x,t)$, then the speed $|\dot q(t)|$ is bounded 
by  $\|K^0_p\|_{{\rm C}^0}$. We also have a bound on
$|\dot p(t)|,\ t\in [0,T]$ in terms of the number of $q_i$ in a ball
${\rm B}_{r(T)}(q(0)),$ with a radius $r(T)$ that depends on $T$ only.
From this we deduce that the corresponding Hamiltonian ODE is well-defined
even though $\rm D^2 H$ is not uniformly bounded.
When $K^0$ is instead of the form $K^0(p)=|p|^2/2$, the corresponding 
Hamiltonian ODE is known as a {\em Lorenz gas} with the following interpretation: $x(t)$ is the state of a particle at time $t$ that is interacting via a potential $V$ with immobile particles at random locations $q_i'$s.
\end{example}

\begin{example}\la{ex2.11}
Let us write $\O$ for the set of discrete subsets of $\bR^{2d}$. 
We also write $\bQ$ for the law of a Poisson point process of intensity one 
on $\O$. On $\O$, we have a
natural translation that is denoted by $\th$: For $\o=\{x_i=(q_i,p_i):\ i\in I\}$,
we define
\[
 \th_x\o=\{x_i-x:\ i\in I\}.
\]
The measure 
$\bQ$ is $\th$ invariant and ergodic. Pick a ${\rm C}^2$ function
$K^0(q,p,t)$ that is of 
compact support in $x=(q,p),$ and $1$-periodic in $t$.
 Given a realization of $\o=\{x_i:\ i\in I\}$, we define
a Hamiltonian function
\[
H(x,t,\o):=\sum_{i\in I}K^0(x-x_i,t).
\]
Observe that this sum is finite $\bQ$-almost surely,
because $K^0$ is of compact support, and $\o$ is discrete.
The map $\o \mapsto H(\cdot,\cdot,\o)$ pushes forward the probability measure
$\bQ$ to a probability measure $\bP$ that is $\th$-invariant and ergodic.
We remark that the Hamiltonian vector field 
$$X(x,t,\o)=J\nabla H(x,t,\o)$$ is not a Lipschitz
map. However we conjecture that one should be able to construct 
a nice flow for $X$, $\bP$-almost surely.
\end{example}

\ms
\begin{remark} As we mentioned earlier, our main results in Section 5, namely Theorems~\ref{th5.1} and
\ref{th5.2}, offer an explicit expression (the formula \eqref{eq5.5})
 for the density of
 $1$-periodic orbits in the setting of Example~\ref{ex2.8}. This formula is based 
on the classical Coarea Formula. We speculate two possible extensions of the work of this article that would allow us to study the other examples we formulated above:
\bi
\item We expect an analogue of formula \eqref{eq5.5} to hold for the density of periodic orbits in the setting of almost periodic Hamiltonian ODEs (Example~\ref{ex2.8}).
To derive such a formula, we need an analogue of coarea formula for the Haar measure of a topological group that can be regarded as an infinite dimensional 
torus $\bT^\i$. 

\item
We also conjecture that our Theorem~\ref{th4.2} is applicable 
to the model we described in Example~\ref{ex2.11}. As we mentioned before,
Kac-Rice type formulas are often stated and verified for Gaussian processes.
Because of this, we can build Hamiltonian functions from Gaussian processes to 
produce examples for which our Theorem~\ref{th4.2} is applicable.
\ei
\end{remark}

 \subsection{Abstract setting and Poisson Structure}

In an equivalent formulation of our results, we start from a probability
space $(\O,\cF,\bP)$, and a group of measurable maps $$\big(\th_a:\ a\in\bR^n
\big)$$ with $$\th_{a+b}=\th_a\circ\th_b,$$ such that
$\bP$ is $\th$-invariant and ergodic. In our probabilistic setting,
$(\O,\cF,\th,\bP)$, plays the role of a symplectic manifold.
Needless to say that we have no tangent bundle to make sense of symplectic
forms on $\O$. However, it is possible to make sense of a Poisson structure 
on $\O$ that is inherited from the standard Poisson structure of $\bR^{2d}$
via the translation $\th$. 

In order to explain this, we first define an (unbounded)
operator $\pmb \nabla$ that is acting on measurable functions on $\O$.
For the domain of the definition of this operator we write
$\mathfrak{H}^1(\bP)$
(see also Definition~\ref{def3.1} below). It consists of functions $$f:\O\to\bR$$ such that 
\bi
\item $f\in \rm L^2(\bR)$, and that the map $x\mapsto f(\th_x\o)$ is differentiable
at $x=0$ for $\bP$-almost $\o$. 
This derivative is denoted by $\pmb\nabla f(\o)$.
\item  The function $\pmb\nabla f$ is in $\rm L^2(\bP)$.
\ei

Given a measurable $K:\O\x \bR\to\bR$, we define the corresponding Hamiltonian function $H$ by $$H(x,t,\o)=K(\th_x\o,t).$$
When $K(\cdot,t)\in \mathfrak{H}^1(\bP)$, and is continuous in time, we can talk about the corresponding {\em Hamiltonian vector field}
\[
X_K(\o,t)=J\pmb \nabla K(\o,t).
\]
Observe that when $J\nabla H(x,t,\o)$ is ${\rm C}^1$, then we can talk about its
flow $\phi_t^{H(\cdot,\o)}(x)$. Using this, we can define a flow
$\pmb \phi^K_t$ on $\o$ in the following manner:
\[
\pmb \phi^K_t(\o)=\th_{x(t,\o)}\o,\ \ \ 
{\text{where }}\ \ \ x(t,\o):=\phi_t^{H(\cdot,\o)}(0).
\]
In some sense, $\pmb \phi^K_t$ is the flow of the Hamiltonian (or rather Poissonian) vector field $X_K$. In order to explain this, we first construct the
{\em Poisson structure} 
$$
\pmb\{\cdot,\cdot\pmb\}:\ \mathfrak{H}^1(\bP)\x\mathfrak{H}^1(\bP)\to \rm L^1(\bP)
$$
on $(\O,\cF,\th,\bP)$ given by
$$
\pmb\{ f, g \pmb\}= J(\pmb \nabla f)\cdot (\pmb\nabla g).
$$
It is degenerate (expect when $\O=\bT^{2d}$) because it is induced from the 
$2d$-dimensional symplectic structure of $\bR^{2d}$ on the possibly infinite dimensional space $\O$. Observe that for a function $f\in\mathfrak{H}^1(\bP)$,
\begin{align*}
\frac {\rm d}{\rm dt}f\big(\pmb \phi^K_t(\o)\big)&=
\frac {\rm d}{\rm dt}f\big(\th_{x(t,\o)}\o\big)
=(\pmb\nabla f)\big(\th_{x(t,\o)}\o\big)\cdot \dot x(t,\o)\\
&=(\pmb\nabla f)\big(\th_{x(t,\o)}\o\big)\cdot J\nabla H( x(t,\o),t,\o)\\
&=(\pmb\nabla f)\big(\th_{x(t,\o)}\o\big)
\cdot J\pmb\nabla K( \th_{x(t,\o)}\o,t)\\
&=\pmb\{K(\cdot,t), f\pmb\}\big(\pmb \phi^K_t(\o)\big).
\end{align*}

\begin{example} In the quasi periodic setting, $\O=\bT^N
=[0,1]^N,\ 0=1$, $\bP$ is the Lebesgue measure, and 
$\th_x\o=\o+Ax \ (\mod 1)$. In this case, 
$$\big(\O,\pmb\{,\cdot,\cdot\pmb\}\big)$$ is a Poisson manifold, with
\[
\pmb\{f,g\pmb\}= \big(AJA^*\big)( \nabla f)\cdot (\nabla g).
\]
Given $K_0:\bT^N\x\bR\to\bR$, the corresponding Hamiltonian ODE vector field
is $\big(AJA^*\big)( \nabla K)(\o,t)$. We refer to the proof of Proposition~\ref{pro6.1}(vii) below for more details.
\end{example}

\section{Existence of stationary generating functions} \label{pre}

In this section we study the existence and regularity of generating functions
associated with $\th'$ stationary twist symplectic maps. Our results require
a Sobolev-type regularity of the symplectic
twist diffeomorphism $\Phi$ that depends on the choice of the stationary measure $\cP$ or $\bP$. The corresponding Sobolev spaces will be defined in the next definition. To ease the notation, we write
$n$ for $2d$.

\begin{definition}\la{def3.1}
\begin{itemize}
\item[(i)]
 Given a $\th$-invariant probability measure $\bP$ on a measure space $\O$,
we define a group of unitary operators
\[
\cT_x:{\rm L}^2(\bP)\to {\rm L}^2(\bP),\ \ \ \ \cT_x f(\o)=f(\th_x\o).
\]
The inner product of the corresponding $\rm L^2(\bP)$
is denoted by $\langle\cdot,\cdot\rangle$. We also write $\bE$ for the expected value with respect to $\bP$: 
$$\bE \ f=\int_\O f\ {\rm d} \bP.$$
The {\em infinitesimal generator} of the group 
$\cT$ is denoted by $\pmb{\nabla}$,
\begin{eqnarray}
\pmb{\nabla}_j f(\o)&=&\lim_{h\to 0}h^{-1}\big(f(\th_{h{\rm e}_j}\o)-f(\o)
\big), \nonumber \\
 \pmb{\nabla}&=&\big(\pmb{\p}_1,\dots,\pmb{\p}_{n}\big), \nonumber
\end{eqnarray}
where $\{{\rm e}_1,\dots,{\rm e}_{n}\}$ is the standard basis of $\bR^{n}$, and
the convergence is with respect to the $\rm L^2(\bP)$ norm.
When  
\[
f=(f_1,\dots,f_n):\O\to\bR^n,
\]
is vector valued, we write $\pmb D f(\o)$ for a matrix whose $j$-th row is $\pmb{\nabla} f_j$.

We write $\mathfrak{H}^1=\mathfrak{H}^1(\bP)$ 
for the domain of $\pmb\nabla$.
Note that when $f\in\mathfrak{H}^1$, then the function
$a\mapsto f(\th_a\o)$ is  differentiable in ${\rm L}^2_{{\rm loc}}(\bR^{n})$.
By Stone's theorem (see for example \cite{La02}), there is a projection-valued measure $\boE({\rm d}\xi)$ such that 
\begin{eqnarray}
\cT_x&=&\int_{\bR^{n}} {\rm e}^{{\rm i}x\cdot \xi}\ \boE({\rm d}\xi), \nonumber \\
\pmb{\nabla}&=&\rm i\int_{\bR^{n}} \xi\ \boE({\rm d}\xi). \nonumber
\end{eqnarray}
 We also write $\mathfrak{H}^{-1}=\mathfrak{H}^{-1}(\bP)$ 
for the domain of the definition
of the operator 
\[
\pmb\nabla^{-1}=-{\rm i}\int_{\bR^{n}} \xi^{-1}\ \boE({\rm d}\xi),
\]
where $\xi^{-1}=(\xi_1^{-1},\dots,\xi_{n}^{-1})$.
If 
$${\rm Z}_f({\rm d}\xi):=\boE({\rm d}\xi)f,\ \ \ \ {\rm G}_f({\rm d}\xi):=
\langle \boE({\rm d}\xi)f,f\rangle,$$
then 
\begin{align}\la{eq3.1}
 f(x,\o):=&f(\th_x\o)=\int_{\bR^{n}} {\rm e}^{{\rm i} x\cdot\xi}\ {\rm Z}_f({\rm d}\xi,\o),\\
{\rm R}_f(x):=&\langle\cT_x f, f\rangle=\int_{\bR^{n}} {\rm e}^{{\rm i}x\cdot\xi}\ {\rm G}_f({\rm d}\xi).\la{eq3.2}
\end{align} 
 From
\[
(- \D_x )^{\pm1} f(x,\o)=
\int_{\bR^{n}} |\xi|^{\pm 2}\  {\rm e}^{{\rm i} x\cdot\xi}\ {\rm Z}_f({\rm d}\xi,\o),
\]
we learn that $f\in \mathfrak{H}^{\pm1}$ if and only if
\[
\int_{\bR^{n}} |\xi|^{\pm 2}\ {\rm G}_f({\rm d}\xi)<\i.
\]
In particular, from
\[
(- \D_x )^{-1} R_f(x)=
\int_{\bR^{n}} |\xi|^{- 2}\  {\rm e}^{{\rm i} x\cdot\xi}\ {\rm G}_f({\rm d}\xi),
\] 
we deduce,
\be\la{eq3.3}
\int_{\bR^{n}} |\xi|^{-2}\ {\rm G}_f({\rm d}\xi)  = (-\D)^{-1}{\rm R}_f(0) =  \int_{\bR^{n}} L(x) {\rm R}_f(x)\ {\rm d}x, 
\ee
where $L:\bR^n\to\bR$ is given by:
\be\la{eq3.4}
L(x)=\begin{cases}(n\a_{n})^{-1}|x|^{2-n}&\ \ \ \ n>2,\\
-(2\pi)^{-1}\log|x|&\ \ \ \ n=2,
\end{cases}
\ee
where $\a_n$ is the $(n-1)$-dimensional surface area of
 the unit sphere $\mathbb{S}^{n-1}$.

\item[(ii)]
We write $\widehat{\mathfrak{H}}^{-1}=
\widehat{\mathfrak{H}}^{-1}(\bP)$ for the set of
$f\in {\rm L}^2(\bP)$ such that 
\[
\int_{|x|\geq 1} |L(x) {\rm R}_f(x)|\ {\rm d}x<\i.
\]
Equivalently,
\[
\int_{\bR^{2d}} |L(x) {\rm R}_f(x)|\ {\rm d}x<\i.
\]
because $|{\rm R}_f|\leq \|f\|_{\rm L^2(\bP)}$, and 
the function $L$ is integrable near $0$. As a consequence,
\[
\widehat{\mathfrak{H}}^{-1}\subseteq 
\mathfrak{H}^{-1}.
\]
\end{itemize}
\end{definition}

\begin{remark}\la{rem3.2}
Assume that $f\in \rm L^2(\bP)$ and $$\int_{\O} f\ \rm d\bP=0,$$
and write $\cL(f)$ for the $\rm L^2(\bP)$-closure of the 
span of the set $\big\{\th_a f:\ a\in \bR^n\big\}$.
The spectral representation \eqref{eq3.1} can be used to 
define a $\rm L^2(\bP)$-isometry between $\cL(f)$ and $\rm L^2(G_f)$:
\[
\cI_f:\rm L^2(G_f)\to \rm L^2(\bP),
\]
so that if $\chi_{x}(\xi)=\rm e^{\rm ix\cdot\xi}$, then 
$$\cI\big(\chi_{x}\big)=\cT_x f,
$$ 
(see for example \cite[Section 5.4]{AT07}.)
To explain this, observe that for any bounded continuous function $\z:\bR^n\to\bR$, we can use \eqref{eq3.1} to write
\[
F_\z(\o):=\int_{\bR^n}\z(x) f(\th_x\o)\ {\rm d}x=\int_{\bR^n}
 \hat\z(\xi) \ Z(\rm d\xi,\o),
\] 
where
\[
\hat\z(\xi)=\int_{\bR^n} \z(x)e^{{\rm i} x\cdot\xi}\ {\rm d}x.
\]
From this, one can show
\[
\bE |F_\z|^2=\int_{\bR^n}\big|\hat\z(\xi) \big|^2\ G_f(\rm d\xi).
\]
Clearly, $\cI_f(\hat\z)=F_\z$. 
\end{remark}

\ms
We continue with some preparatory definitions regarding stationary functions and twist maps.

\begin{definition} \label{gdd}
\begin{itemize}
\item[(i)]
Let us write $\cH$ for the space of  ${\rm C}^2$ Hamiltonian functions
$H:\bR^{2d}\x\bR\to\bR$. For each $a=(b,c)\in\bR^{d}\x\bR^d$, we define
\begin{eqnarray}
(\tau_bH)(q,p,t)&=&H(q+b,p,t), \nonumber \\
(\eta_cH)(q,p,t)&=&H(q,p+c,t), \nonumber \\
(\th_aH)(q,p,t)&=&H(q+b,p+c,t). \nonumber
\end{eqnarray}

\item[(ii)] We write $\cC^1$ for the set of ${\rm C}^1$ maps $\Phi:
\bR^{2d}\to\bR^{2d}$.
We set $$\cF(\Phi)=\Phi-{\rm id},$$ where ${\rm id}$ denotes the identity map. 

We write $\cS$ for the set of symplectic diffeomorphisms $\Phi:\bR^{2d}\to\bR^{2d}$ such that $\cF(\Phi)$ is uniformly bounded. 
We also set $\tilde \cS=\cF\big(\cS\big)$.  

 For $a\in\bR^{2d}$, the translation operators $\th_a:\bR^{2d}\to\bR^{2d}$  and
$\th_a,\th'_a:\cC^1\to\cC^1$ are defined by
\begin{eqnarray}
\th_a (x)&=&x+a, \nonumber \\
\th_a\o&=&\o\circ \th_a, \nonumber \\
\th'_a&=&\cF^{-1}\circ \th_a\circ \cF, \nonumber
\end{eqnarray}
for $x\in\bR^{2d}$ and $\o\in\cC^1$. Note that for $\Phi\in\cC^1,$
\begin{eqnarray}
\big(\th'_a \Phi\big)(x)&=&(\th_{-a}\circ \Phi\circ \th_a)(x) \nonumber \\
&=&\Phi(x+a)-a. \nonumber
\end{eqnarray}

\item[(iii)] Let $\cP$ be a $\th'$-invariant probability measure on $\cS$.
The map $\cF$ pushes forward to a measure on $\tilde S$ that is denoted by 
$\cQ$.
This measure is $\th$ invariant.

\item[(iv)] We define $\pi:\tilde\cS\to\bR$ to be the evaluation map 
$\pi(\o)=\o(0)=\Phi(0)$.

\item[(v)] Let  $\Phi$ be a symplectic diffeomorphism with 
\[
\Phi(q,p)=\big(Q(q,p),P(q,p)\big). 
\]
We say that $\Phi$ is {\em twist} if for every $p\in\bR^d$, the map $q\mapsto Q(q,p)$ is a diffeomorphism
of $\bR^d$. We write $\hat q(Q,p)$ for the inverse:
\[
Q(q,p)=Q\ \ \ \ \Longleftrightarrow \ \ \ \ q=\hat q(Q,p).
\]
We also set $\hat P(Q,p)=P\big(\hat q(Q,p),p)$, and
\begin{eqnarray}
\widehat\Phi(Q,p)&=&\big(\hat q(Q,p),\hat P(Q,p)\big), \nonumber \\
\widetilde\Phi(Q,p)&=&\big(\hat P(Q,p),\hat q(Q,p)\big).  \nonumber
\end{eqnarray}
\end{itemize}
\end{definition}

We are now ready to state the main result of this section.

\begin{theorem} \la{th3.3}  
 Let $\cQ$ be a $\th$-invariant measure such that 
$\pi\in\widehat{\mathfrak{H}}^{-1}(\cQ)$, 
\be\la{eq3.5}
\int_{\cS} \Phi(0)\ \cP({\rm d}\Phi)=\int_{\tilde\cS} \o(0)\ \cQ({\rm d}\o)=0,
\ee
and
\be\la{eq3.6}
 \int_{\tilde\cS} 
\|{\rm D}\o\|_{\rm C^0}^{d}\ \cQ({\rm d}\o)<\i.
\ee
Assume that $\Phi={\rm id}+\o$ is ${\rm C}^2$ twist diffeomorphism
 $\cQ$-almost surely. Then there exists
a unique function $\hat w:\tilde\cS\to\bR$, with $\hat w\in \rm L^2(\cQ)$, and
\be\la{eq3.7}
\int_{\tilde\cS} \hat w(\o)\ \cQ({\rm d}\o)=0,
\ee
such that if 
\[
w(x,\o):=\hat w\big(\th_x \o\big),\ \ \ \ 
W(Q,p,\o)=Q\cdot p+w(Q,p,\o),
\]
then 
 $$\widehat\Phi=(W_p,W_Q)=:\widehat \nabla W,$$
$\cQ$-almost surely.
\end{theorem}

\begin{remark}
Note that $\widehat \nabla =\big(\p_p,\p_Q\big)$ 
represents the gradient operator with $\p_p$ and $\p_Q$ swapped.
We refer to $\hat w$ of Theorem~\ref{th3.3} as a {\em stationary generating function}. According to this theorem,
a $\th'$ stationary symplectic twist map always possesses a stationary generating function. A natural question is whether the converse is true.
 Given a function $\hat w$ such that the corresponding stationary
process 
\[
w(x)=w(x,\o):=\hat w\big(\th_x \o\big),
\]
 is ${\rm C}^2$, can
we use this function to produce a symplectic $\th'$-stationary twist map
$\Phi$?  This is equivalent to the condition that
\[
Q\mapsto W_p(Q,p)=Q+w_p(Q,p),
\]
is a diffeomorphism for each $p$, so that we can solve
the equation $$W_p(Q,p)=q$$ for $Q= Q(q,p)$. This is always possible if the ${\rm C}^2$ norm of $w$ is small 
(see Proposition~\ref{pro3.1}(iv)). Moreover, when $d=1$, we need 
\begin{align*}
&W_{pQ}=1+w_{pQ}>0,\ \ \ \ W_p(\pm \i,p)=\pm\i,\ \ \ {\text{or}}\\
&W_{pQ}=1+w_{pQ}<0,\ \ \ \ W_p(\pm \i,p)=\mp\i.
\end{align*}
The latter condition can be guaranteed by assuming that $w_p$ is a bounded function.
\end{remark}

\ms
With the previous definitions in mind, we state and prove  three preparatory
propositions.
\begin{proposition} \label{pro3.1}
The following statements hold.
\begin{itemize}
\item[(i)] For every symplectic twist diffeomorphism 
$\Phi:\bR^{2d}\to\bR^{2d}$,  and $a\in\bR^d$, we have 
\[
\widehat{\th'_a\Phi}=\th'_a\widehat{\Phi}.
\]

\item[(ii)] If $\pi\in \widehat{\mathfrak{H}}^{-1}(\cQ)$, 
and 
\[
\int_{\cS} \|{\rm D}\Phi\|_{\rm C^0}^d\ \cP(\rm d\Phi)=
\int_{\tilde\cS}\|{\rm D}\o\|_{\rm C^0}^d\ \cQ(\rm d\o)<\i,
\]
 then 
$$\hat\pi\in \widehat{\mathfrak{H}}^{-1}(\cQ),$$
where $\hat\pi(\o):=\widehat\Phi(0)$.

\item[(iii)] Assume that $\Phi\in\cS$ is symplectic twist diffeomorphism. Then there exists a ${\rm C}^2$ function 
$$W \colon \bR^{2d}\to\bR$$ such that $\widetilde\Phi=\nabla W$.
Moreover, for $w$ defined by $$w(Q,p):=W(Q,p)-Q\cdot p,$$
we have 
\begin{eqnarray}  \label{eq3.8}
\|\nabla w\|_{\rm C^0}\leq \|\cF(\Phi)\|_{\rm C^0}.
\end{eqnarray}

\item[(iv)] Let $w$ be a ${\rm C}^2$ function with $\|{\rm D}^2 w\|<1$, and set

\be\la{eq3.9}
W(Q,p)=Q\cdot p+w(Q,p).
\ee
 Then there exists a symplectic twist diffeomorphism $\Phi$ such that 
\[
\Phi\big(W_p(Q,p), p\big)=\big(Q,W_Q(Q,p)\big).
\]
\end{itemize}

\end{proposition}

\begin{proof}
{(i)} Let us write
\begin{eqnarray}
\Phi'(q,p)&:=&(\th'_a\Phi)(q,p)=\big(Q'(q,p),P'(q,p)\big), \nonumber \\
  \widehat\Phi'(Q,p)&=&\big(\hat q'(Q,p),\hat P'(Q,p)\big). \nonumber
\end{eqnarray}
This  implies
\begin{align*}
&Q'(q,p)=Q(q+b,p+c)-b=Q\ \ \ \ \Longleftrightarrow \ \ \ \
\hat q'(Q,p)=q,\\
&Q(q+b,p+c)=Q+b\ \ \ \ \ \ \ \ \ \ \ \ \ \ \ \ \ \ \Longleftrightarrow \ \ \ \
\hat q(Q+b,p+c)=q+b.
\end{align*}
Hence 
\[
\hat q'(Q,p)=\hat q(Q+b,p+c)-b.
\]
On the other hand
\begin{eqnarray}
\hat P'(Q,p)&=&P'\big(\hat q'(Q,p),p\big) \nonumber \\
&=& P\big(\hat q'(Q,p)+b,p+c\big)-c \nonumber \\
&=& P\big(\hat q(Q+b,p+c),p+c\big)-c \nonumber \\
&=& \hat P\big(Q+b,p+c\big)-c, \nonumber
\end{eqnarray}
as desired.

\ms\noi
{(ii)} To ease the notation, we write $\bE$ for the $\rm d\cQ$ integration.
Fix some $x^0=(q^0,p^0)\in\bR^{2d}$.
Our claim reads as follows: If
\be\la{eq3.10}
\bE\int_{|x-x^0|\geq 1}L(x-x^0)\big|(\Phi(x)-x)\cdot (\Phi(x^0)-x^0)\big|\ {\rm d}x<\i,
\ee
then
\be\la{eq3.11}
\bE\int_{|x-x^0|\geq 1}L(x-x^0)\big|(\widehat\Phi(x)-x)\cdot 
(\widehat\Phi(x^0)-x^0)\big|\ {\rm d}x<\i.
\ee
By stationarity, neither the statement \eqref{eq3.10} nor the statement
\eqref{eq3.11} depend on the choice of the point $x^0$.
To ease the notation, we write 
\begin{eqnarray}
X&=&(Q,P)=\Phi(x)=\Phi(q,p), \nonumber \\
X'&=&(Q,p). \nonumber
\end{eqnarray}
Write $y^0=(Q^0,p^0)$, where $Q^0=Q(x^0)$. Since 
\[
(\Phi(x)-x)\cdot (\Phi(x^0)-x^0)=\left(\widehat \Phi(X')-X'\right)\cdot 
\left(\widehat \Phi(y^0 )-y^0\right),
\]
the statement~\eqref{eq3.10} can be rewritten as
\be\la{eq3.12}
\bE\int_{|x-x^0|\geq 1}L(x-x^0)\big|(\widehat\Phi(X')-X')\cdot 
(\widehat\Phi(y^0)-y^0)\big|\ {\rm d}q \ {\rm d}p<\i.
\ee
Assuming this, we wish to show
\be\la{eq3.13}
\bE\int_{|X'-y^0|\geq 1}L(X'-y^0)\big|(\widehat\Phi(X')-X')\cdot 
(\widehat\Phi(y^0)-y^0)\big|\ {\rm d}Q \ {\rm d}p<\i.
\ee
Since the law of 
both $$x\mapsto \Phi(x)-x,$$ and $$X'\mapsto\widehat \Phi(X')-X'$$
are $\th$-invariant, we can choose $x^0=0$ in \eqref{eq3.12}, and
$y^0=0$ in \eqref{eq3.13}.

We first assume that $d>1$. 
 Observe that if
$c_0\geq \sup|\cF(\Phi)|$, then $$|X'-x|\leq c_0,$$ and 
\begin{align}\nonumber
|x|\geq 2c_0\ \ \ \ &\implies\ \ \   \frac 12 |x|\leq  |X'|\leq \frac 32 |x|,\\
|X'|\geq 3c_0\ \ \ \ &\implies\ \ \   \frac 23 |X'|\leq  |x|\leq \frac 43 |X'|.\la{eq3.14}
\end{align}
This in turn implies under the assumption  $|x|\geq 2c_0$,
\begin{align*}
   \big||X'|^{-r}- |x|^{-r}\big|&=|X'|^{-r} |x|^{-r}
\big||X'|^{r}- |x|^{r}\big|\\
&=(|X'||x|)^{-r}
\big||X'|- |x|\big|\sum_{j=0}^{r-1}|X'|^j|x|^{r-1-j}\\
&\leq c_1c_0 |x|^{-r-1},
\end{align*}
for a constant $c_1=c_1(r)$ that depends on $r=2d-2$ only.
We additionally require $c_0\geq 1/2$ so that $2c_0\geq 1$. Hence, for
\[
\L:=\int_{|x|\geq 2c_0} \big|(\widehat\Phi(X')-X')\cdot 
(\widehat\Phi(y^0)-y^0)\big|\ |X'|^{-r}\ {\rm d}x,
\]
we have
\begin{align*}
\L= &\int_{|x|\geq 2c_0} \big|(\widehat\Phi(X')-X')\cdot 
(\widehat\Phi(y^0)-y^0)\big|\ |X'|^{-r}\ {\rm d}x\\
\leq &\int_{|x|\geq 2c_0} \big|(\widehat\Phi(X')-X')\cdot 
(\widehat\Phi(y^0)-y^0)\big|\ |x|^{-r}\ {\rm d}x\\
&+c_1 c_0\int_{|x|\geq 2c_0} \big|(\widehat\Phi(X')-X')\cdot 
(\widehat\Phi(y^0)-y^0)\big|\ |x|^{-r-1}\ {\rm d}x\\
=&:\L_0+\L_1,
\end{align*}
where $X^0=\Phi(x^0)$. The assumption \eqref{eq3.12} implies then 
$$\bE(\L_0+\L_1)<\i.$$ Hence 
\be\la{eq3.15}
\bE\L<\i.
\ee 
We now make a change of variables to replace
$q$ with $\hat q(Q,p)$ in $\L$. Note
\begin{eqnarray}
&& {\rm d}Q\ {\rm d}p=\big|\det  Q_q(q,p)\big|\  {\rm d}q\ {\rm d}p, \nonumber \\
&& \big|\det Q_q(q,p)\big|\leq d! |{\rm D}\Phi|^d=:c_2. \la{eq3.16}
\end{eqnarray}
From this, \eqref{eq3.14}, and \eqref{eq3.15} we learn
\begin{eqnarray}
&&\bE \int_{|X'|\geq 3c_0} \big|(\widehat\Phi(X')-X')\cdot 
(\widehat\Phi(y^0)-y^0)\big|\ |X'|^{-r}\ {\rm d}Q \ {\rm d}p  \nonumber \\
&\leq & \bE\int_{|x|\geq 2c_0} \big|(\widehat\Phi(X')-X')\cdot 
(\widehat\Phi(y^0)-y^0)\big|\ |X'|^{-r}\ {\rm d}Q \ {\rm d}p  \nonumber \\
&\leq& c_2\int_{|x|\geq 2c_0} \big|(\widehat\Phi(X')-X')\cdot 
(\widehat\Phi(y^0)-y^0)\big|\ |X'|^{-r}\ {\rm d}q \ {\rm d}p \nonumber \\
&=&c_3\bE\L<\i. \nonumber
\end{eqnarray}
Because of this, the claim \eqref{eq3.13} (in the case $y_0=0$) 
would follow if we can show
\[
\bE\int_{1\le|X'|\leq 3c_0} \left|\left(\widehat \Phi(X')-X'\right)\cdot 
\left(\widehat \Phi(y^0 )-y^0\right)\right|\ |X'|^{-r}\ {\rm d}Q \ {\rm d}p<\i,
\]
Since the law of $\widehat\Phi(X')-X'$ is $\th$-stationary, we can bound the left-hand side by
\[
 \bE\big|\widehat \Phi(0)\big|^2\ 
\int_{1\le|X'|\leq 3c_0} \ |X'|^{-r}\ {\rm d}Q \ {\rm d}p=:c_3  \bE\big|\widehat \Phi(0)\big|^2.
\]
It remains to verify
\be\la{eq3.17}
 \bE\big|\widehat \Phi(0)\big|^2<\i.
\ee
Indeed from $|X'-x|\leq  c_0$, \eqref{eq3.15}, and the stationarity, we deduce
\begin{align*}
\bE\big|\widehat \Phi(0)\big|^2=&
\frac 1{|{\rm B}_1(0)|}\ \bE\int_{{\rm B}_1(0)}\big|\widehat \Phi(X')-X'\big|^2 \ {\rm d} Q \ {\rm d}p\\
\leq & c_2
\frac 1{|{\rm B}_1(0)|}\ \bE\int_{{\rm B}_{1+c_0}(0)}\big|\widehat \Phi(X')-X'\big|^2 \ {\rm d}q \ {\rm d}p
\\ = & c_2
\frac 1{|{\rm B}_1(0)|}\ \bE\int_{{\rm B}_{1+c_0}(0)}\big| \Phi(x)-x\big|^2 \ {\rm d}q \ {\rm d}p\\
= & c_2
\frac {|{\rm B}_{1+c_0}(0)|}{|{\rm B}_1(0)|}\ \bE| \Phi(0)\big|^2 <\i,
\end{align*}
where we used 
\[
\big|\widehat \Phi(X')-X'\big|^2 =\big| \Phi(x)-x\big|^2.
\]
for the second equality. The proof is complete when $d>1$.

The proof in the case $d=1$ is similar. Observe that from
\begin{eqnarray}
\big|\log|X'|-\log|x|\big|&=&\left|\int_{|x|}^{|X'|}\frac {{\rm d}r}r\right| \nonumber \\
&\leq& \frac{|X'-x|}{|X'|\wedge |x|}, \nonumber
\end{eqnarray}
and $|X'-x|\leq  c_0$, we deduce
\begin{eqnarray}
\max\{|x|,|X'|\}\geq 2c_0 &\implies& \min\{|x|,|X'|\}\geq c_0 \nonumber \\
&\implies& \big|\log|X'|-\log|x|\big|\leq 1. \nonumber
\end{eqnarray}
This would allow us to repeat our proof for the case $d>1$ and finish the proof.

\ms\noi
{(iii)}
Since $\Phi$ is symplectic, we have 
\begin{eqnarray}
0&=&{\rm d}\big( P\cdot {\rm d}Q-p\cdot  {\rm d}q) \nonumber \\
&=&{\rm d}\big(\hat P\cdot {\rm d}Q-p\cdot {\rm d}\hat q) \nonumber \\
&=&{\rm d}\big(\hat P\cdot {\rm d}Q+\hat q\cdot {\rm d}p). \nonumber
\end{eqnarray}
Hence, there exists a function $W=W(Q,p)$ such that 
\[
{\rm d}W=\hat P\cdot {\rm d}Q+\hat q\cdot {\rm d}p.
\]
As a result, $$\nabla W=\widetilde \Phi.$$

 The inequality \eqref{eq3.8} is an immediate consequence of 
\begin{eqnarray}
\cF(\Phi)(q,p)&=&(Q-q,P-p) \nonumber \\
&=& (Q-W_p(Q,p),W_Q(Q,p)-p) \nonumber \\
&=&(-w_p(Q,p),w_Q(Q,p)). \nonumber
\end{eqnarray}

\ms\noi
{(iv)}
 If we define
\begin{eqnarray}
\widehat \Phi(Q,p)&=&\big(W_p(Q,p),W_Q(Q,p)\big) \nonumber \\
&=&(Q,p)+\big(w_p(Q,p),w_Q(Q,p)\big),  \nonumber
\end{eqnarray}
then $\widehat \Phi$ is a ${\rm C}^1$ diffeomorphism by our assumption
on $w$. In particular, the equation
$W_p(Q,p)=q$ ,
can be solved implicitly for $Q=Q(q,p)$.
We may define $$P(q,p)=W_Q\big(Q(q,p),p\big),$$ and 
$$\Phi(q,p)=\big(Q(q,p),P(q,p)\big),$$
which concludes the proof.
\end{proof}

\ms
As we have learned from Proposition~\ref{pro3.1} (iii), a symplectic twist diffeomorphism 
always has a generating function. What  
Theorem~\ref{th3.3} claims is the existence of 
a {\em stationary} generating function. Note that if we set
\be\la{eq3.19}
B(Q,p)=B(Q,q,\o):=\widehat\Phi(Q,p)-(Q,p),
\ee
then $B$ is a $\th$-stationary by Proposition~\ref{pro3.1}(i).
By Proposition~\ref{pro3.1}(iii), we can express $B$ as $\widehat \nabla w$
for a function $w(Q,p)=w(Q,p,\o)$. We wish to show that this $w$  
can be chosen to be a stationary process with respect to $\th$. 
In the next Proposition, 
we state a sufficient condition (which is also necessary) for the existence of such stationary generating function.

\begin{proposition}  \la{pro3.2}  
Let $\cQ$ be a $\th$-invariant probability measure on
$\O$, and let 
$$\hat B:\O\to\bR^{2d}$$ be a function 
with the following properties:
\begin{itemize}
\item[(i)] $\hat B\in {\mathfrak{H}}^{-1}(\cQ)$, and 
\be\la{eq3.20}
\int_\O \hat B(\o)\  \cQ({\rm d}\o)=0.
\ee
\item[(ii)]  There exists a ${\rm C}^2$ function 
$v(x,\o)$ such that 
\[
B(x,\o):=\hat B(\th_x\o)=\widehat\nabla v(x,\o).
\]
\end{itemize}
Then there exists a unique $\hat w\in \rm L^2(\cQ)$ such that 
\[
\int_\O \hat w\ \rm d\cQ=0,
\]
and if $w(x,\o):=\hat w(\th_x \o),$ then
 $\widehat \nabla w=B,$ $\cQ$-almost surely.
\end{proposition}

\begin{proof} Since the process $B$ is stationary, 
by the Spectral Theorem \eqref{eq3.1}, 
we can find a vector measure $$Z({\rm d}\xi,\o)=(Z_j({\rm d}\xi,\o):\ j=1,\dots,2d)$$ such that 
\begin{eqnarray}
B(x,\o)&=& \int_{\bR^{2d}} {\rm e}^{{\rm i}x\cdot\xi}\ Z({\rm d}\xi,\o), \nonumber \\
 {\rm e}^{{\rm i}a\cdot\xi}Z({\rm d}\xi,\o) &=&
Z({\rm d}\xi,\th_a\o). \nonumber
\end{eqnarray}
Let us write $\eta:\bR^{2d}\to\bR^{2d}$ for the function
that swaps $Q$ with $p$: 
\[
\eta(Q,p)=(p,Q).
\]
 If we write
$B'$ for $\eta(B)$, and $Z'$ for $\eta(Z)$, then 
\[
B'(x,\o)= \int_{\bR^{2d}} {\rm e}^{{\rm i}x\cdot\xi}\ Z'({\rm d}\xi,\o).
\]
Since $B=\widehat\nabla v$, for some function $v$, we have
$B'=\nabla v$ is an exact derivative. This means
\[
{\rm D} B'(x,\o)={\rm i}\left[\int_{\bR^{2d}}  {\rm e}^{{\rm i}x\cdot\xi}\ \xi_jZ'_k({\rm d}\xi,\o)\right]_{ j,k=1}^{2d},
\]
is a symmetric matrix. As a result, 
$$\xi_jZ'_k({\rm d}\xi,\o)=\xi_kZ'_j({\rm d}\xi,\o),$$
which in turn implies that the scalar measure 
\[
z({\rm d}\xi,\o)=\xi_j^{-1}Z'_j({\rm d}\xi,\o),
\]
is independent of $j$. In summary,
\begin{eqnarray}
Z'({\rm d}\xi,\o) &=& \xi\ z({\rm d}\xi,\o), \nonumber \\
 {\rm e}^{{\rm i}a\cdot\xi}z({\rm d}\xi,\o)&=& z({\rm d}\xi,\th_a\o). \la{eq3.21}
\end{eqnarray}
Hence
\be\la{eq3.22}
Z({\rm d}\xi,\o) = \eta(\xi)\ z({\rm d}\xi,\o).
\ee
Our candidate for $\hat w$ is simply
\[
\hat w(\o):=-{\rm i}z\big(\bR^{2d},\o\big).
\]
We claim that our assumption $\hat B\in {\mathfrak{H}}^{-1}(\cQ)$
guarantees that $\hat w$ is well-defined and $\hat w\in \rm L^2(\cQ)$.
Once this is done, we can then use \eqref{eq3.21} to deduce
\be\la{eq3.23}
 w(x,\o):=w(\th_x\o)=-{\rm i}\int_{\bR^{2d}}  {\rm e}^{{\rm i}x\cdot\xi}\ z({\rm d}\xi,\o).
\ee

  Recall that if
\[
R(a):=\bE\ B(0,\o)\otimes \bar B(0,\th_a\o)
\]
represents the correlation of $B$,  then by \eqref{eq3.2},
\[
R(a)=\int_{\bR^{2d}}  {\rm e}^{{\rm i} \xi\cdot a}\ G_B({\rm d}\xi),\ \ {\text{where}}\ \ 
G_B({\rm d}\xi)=\bE\ Z({\rm d}\xi,\o)\otimes \bar Z({\rm d}\xi,\o).
\]
From this, and \eqref{eq3.22} we learn,
\begin{eqnarray}
G_B({\rm d}\xi)&=&\eta(\xi)\otimes \overline{\eta(\xi)} \ \bE \ |z|^2({\rm d}\xi,\o)  \nonumber \\
&=:& \eta(\xi)\otimes \overline{\eta(\xi)}\ g({\rm d}\xi)\nonumber \\
&=:&\big[G_B^{jk}(\rm d\xi)\big]_{j,k=1}^{2d}.\nonumber
\end{eqnarray}
Since $\hat B\in {\mathfrak{H}}^{-1}(\cQ)$, we have 
\[
\int_{\bR^{2d}}|\xi|^{-2}\ G_B(d\xi)<\i.
\]
This means that the map $\xi\mapsto\z_j(\xi):=\xi_j^{-1}$ is in $\rm L^2(G)$. 
We then use the isometry $\cI_B$ of Remark~\ref{rem3.2} to assert
that $\hat w=\cI_B(\z_j)$ is in $\rm L^2(\cQ)$. Moreover, if
\[
r(a):=\bE\ w(0,\o) \bar w(0,\th_a\o),
\]
represents the correlation of $ w$,
then
\[
 r(a)=\int_{\bR^{2d}}  {\rm e}^{{\rm i} \xi\cdot a}\ g({\rm d}\xi).
\]
In particular,
\[
\bE |\hat w|^2=r(0)=g(\bR^d) = \int_{\bR^{2d}}  |\xi_{j}|^{-2}\ G_B^{jj}({\rm d}\xi)<\i, 
\]
for any $j\in\{1,\dots,2d\}$. From differentiating \eqref{eq3.23}, 
we can readily deduce that $\widehat\nabla w=B$ weakly. Since $B$ is ${\rm C}^1$,
we conclude that $\hat w\in {\rm C}^2$.

It remains to verify the uniqueness of $\hat w$.
Note that if $\hat w'\in \rm L^2(\cQ)$ such that
\[
\int \hat w'\ \rm d\cQ=0,
\]
 and the corresponding $w'$ is ${\rm C}^1$
function satisfying $\widehat \nabla w' =B$, then $\hat\z=\hat w-\hat w'$ satisfies
$\widehat\nabla \z=0$, for $\z(x,\o)=\hat\z(\th_x\o)$. 
This means that $\hat \z(\th_x \o)$ does not depend on $x$.
Since the measure $\cQ$ is ergodic with respect to $\th$, we deduce that
$\hat \z$ is constant $\cQ$-almost surely. Since $\cQ$-integral of $\hat z$ is zero,
we deduce that $\hat z=0$. Hence $\hat w=\hat w'$, proving the uniqueness
of $w$.
\end{proof}

\begin{example} \label{Example3.6}
In this example, we examine the set $\mathfrak H^{-1}$
in the setting of quasiperiodic functions (see Example~2.8).
On the torus $\O=[0,1]^{N},\ 0=1$, we define the flow $$\th_x \o=\o+Ax \mod 1,$$ where $A$ is a 
$N\x(2d)$ matrix, and
$x\in\bR^{2d}$. Recall that if $\bP$ denotes the Lebesgue measure on $\O$, then
$\bP$ is ergodic with respect to $\th$ if and only if 
\[
m\in \bZ^{N}\setminus\{0\}\ \ \ \implies\ \ \ mA\neq 0.
\]
Consider the function
\begin{eqnarray}
u(\o)&=&\sum_{m\in\bZ^{N}}a_m{\rm e}^{{\rm i} m\cdot \o},\nonumber \\
u(\th_x\o)&=&\sum_{m\in\bZ^{N}}a_m{\rm e}^{{\rm i}(mA)\cdot x}{\rm e}^{{\rm i}m\cdot \o}. \nonumber
\end{eqnarray}
From this and 
\[
\bE \ u(\th_x\o)\overline {u(\o)}=\sum_{m\in\bZ^{N}}|a_m|^2{\rm e}^{{\rm i}(mA)\cdot x},
\]
we deduce, 
\begin{eqnarray}
Z({\rm d}\xi,\o)&=&
\sum_{m\in\bZ^N}a_m{\rm e}^{{\rm i}m\cdot\o}\ \d_{mA}({\rm d}\xi),\nonumber \\
G({\rm d}\xi)&=&\sum_{m\in\bZ^N}|a_m|^2\d_{mA}({\rm d}\xi),\nonumber\\
g({\rm d}\xi)&=&\sum_{m\in\bZ^N}|a_m|^2|mA|^{-2}\d_{mA}({\rm d}\xi). \nonumber
\end{eqnarray}
Hence 
\[
g(\bR^{2d})=\sum_{m\in\bZ^N}|a_m|^2|mA|^{-2}.
\]
The function $u\in \mathfrak H^{-1}(\bP)$ if $g(\bR^{2d})<\i$.
For example, a {\em Diophantine condition} of the form
\[
m\in \bZ^{N}\setminus\{0\}\ \ \ \implies\ \ \ |mA|\geq |m|^{-k},
\]
yields 
\[
\sum_{m\in\bZ^N}|m|^{2k}|a_m|^2<\i\ \ \ \implies\ \ \ g(\bR^{2d})<\i.
\]
Hence if $u$ possesses $k$ many derivatives in ${\rm L}^2$, then 
$u\in \mathfrak H^{-1}(\bP)$.
\end{example}

\ms
Our next ingredient for the proof of Theorem~\ref{th3.3} is an application of 
Ergodic Theorem.

 \bs
\begin{proposition} \la{pro3.3}
 Let $v(x)=v(x,\o)=\hat v(\tau_x\o)$ be a stationary process 
with $$c_0:=\bE |\hat v|<\i.$$ Given $\ell=(\ell_1,\dots,\ell_{2d})$, write
\[
{\rm I}(\ell)=\prod_{i=1}^{2d}[-\ell_i,\ell_i].
\]
Then almost surely, we can finds a sequence $$\ell^r=(\ell_1^r,\dots,
\ell^r_{2d})$$ such that $\ell_i^r\to\i$ in large $r$ limit, and
\[
\sup_r\s(\p {\rm I}(\ell^r))^{-1}\int_{\p {\rm I}(\ell^r)}|v(x,\o)|\ 
\s({\rm d}x)<\i,
\]
where $\s$ denotes the $2d-1$-dimensional surface measure.
\end{proposition}

\begin{proof}
 To ease the notation, we write $h=|v|$.
Given $r>0$, write ${\rm I}_r$ for $[-r,r]^{2d}$, and define
\[
{\rm M}(\o)=\sup_{r\geq 1}|{\rm I}_r|^{-1}\int_{{\rm I}_r}h(x,\o)\ {\rm d}x.
\]
By the Maximal Ergodic Theorem (see for example Theorem~1.4 of \cite{R}),
\[
\bP(A_s):=\bP\left(\left\{\o:\ {\rm M}(\o)>s\right\}\right)\leq s^{-1}c_0.
\]
Fix a large $s>0$, and $\o\in A_s^c$, so that $${\rm M}(\o)\leq s.$$
Set $\hat x=(x_2,\dots,x_{2d}).$  Since $\o\in A_s^c$, we can write
\begin{eqnarray}
\int_{{\rm I}_r}w(\th_x\o)\ {\rm d}x&=& \int_{0}^rJ_1^r(x_1)\ {\rm d}x_1 \nonumber \\
&:=& \int_{0}^r\left[\int_{\hat{\rm  I}_r} (w(x_1,\hat x)+w(-x_1,\hat x))\ {\rm d}\hat x\right]\ {\rm d}x_1 \nonumber \\
&\leq& s|{\rm I}_r| = s(2r)^{2d}, \nonumber 
\end{eqnarray}
where $\hat {\rm I}_r=[-r,r]^{2d-1}$. From this and 
Chebyshev's inequality, 
\[
\left|\left\{x_1\in[0,r]:\ J_1^r(x_1)>4s(2r)^{2d-1}\right\}\right|\leq
 \frac r2.
\]
As a result, there exists $\ell_1^r\in[r/3,r]$ such that 
\[
J_1^{r}(\ell_1^r)\leq 4s(2r)^{2d-1}.
\]
In the same fashion, we can write 
\[
\int_{\rm I_r}h(x,\o)\ {\rm d}x=\int_{0}^rJ_i^r(x_i)\ {\rm d}x_i,
\]
for $i\in\{1,\dots,2d-1\}$, and find $\ell_i^r\in[r/3,r]$ such that 
\[
J_i^{r}(\ell_i^r)\leq 4s(2r)^{2d-1}.
\]

For $\o\in A_s^c$, and
$\ell^r=\ell_r(\o)=(\ell_1^r,\dots,\ell_{2d}^r)$ as above,
 observe that for each $j$,
\begin{eqnarray}
\int_{\p {\rm I}(\ell^r)}h(x,\o)\ \s({\rm d}x)
&=& \sum_{i=1}^{2d} J_i^{r}(\ell_i^r) \nonumber \\
&\leq& 8ds(2r)^{2d-1} \nonumber \\
&\leq& 8ds3^{2d-1}\prod_{i\neq j}^{2d}(2\ell^r_i) , \nonumber
\end{eqnarray}
because $\ell_i\in[r/3,r]$ for every $i$.
From this we learn
\begin{eqnarray}
\int_{\p {\rm I}(\ell^r)}h(x,\o)\ \s({\rm d}x)
&\leq& 4s3^{2d-1}\sum_{j=1}^{2d}\prod_{i\neq j}^{2d}(2\ell^r_i) \nonumber \\
&=&2s3^{2d-1} \s(\p{\rm I}(\ell^r)). \nonumber
\end{eqnarray}
This completes the proof for $\o\in A_s^c$. Since $\bP(A_s)\to 0$
as $s\to\i$, we are done.
\end{proof}

\ms
With the aid of Propositions~\ref{pro3.1}-\ref{pro3.3}, we are now ready to tackle 
Theorem~\ref{th3.3}.

\subsection*{Proof of Theorem~\ref{th3.3}} {\em(Step 1)}
Recall the process $B$ that was defined in
\eqref{eq3.19}. By Proposition~\ref{pro3.1}(i), the process 
$B$ is stationary. We are done if we can apply Proposition~\ref{pro3.2} to
$B$. For this, we need to verify the properties (i) and (ii) of this Proposition.
Proposition~\ref{pro3.1}(iii) verifies property (ii). The property (i)
consists of two condition. The first condition of this property requires $\hat A$ to be in 
$\mathfrak{H}^{-1}(\cQ)$. For this, it suffices to show
$\hat A\in\widehat{\mathfrak{H}}^{-1}(\cQ)$, which is an immediate consequence
of Proposition~\ref{pro3.1}(ii), and our assumptions $\pi\in
\widehat{\mathfrak{H}}^{-1}(\cQ)$ and \eqref{eq3.6}. 
It remains to verify \eqref{eq3.20}:
\be\la{eq3.24}
a=(b,c):=\int_{\cS}B(0)\ \cP({\rm d}\Phi)=\int_{\cS}\widehat\Phi(0)\ 
\cP({\rm d}\Phi)=0.
\ee
Observe that by Proposition~\ref{pro3.2} is applicable to
$B(Q,p)-a$. In other words,  there exists a ${\rm C}^2$ stationary function $w(Q,p)=\hat w(\th_{(Q,p)}\o)$ such that 
$$B(Q,p)-a=\widehat\nabla w,$$ or
\begin{eqnarray}
\hat P(Q,p)&=&\hat P(Q,p,\o)=c+p+w_Q(Q,p,\o),\nonumber \\
\hat q(Q,p)&=&\hat q(Q,p,\o)=b+Q+w_p(Q,p,\o). \nonumber
\end{eqnarray}
Recall that $\o=\cF(\Phi)=\Phi-{\rm id}$.

\ms\noi
{\em(Step 2)} 
Using our assumption \eqref{eq3.5}
 and the stationarity of $\cQ$, 
\begin{eqnarray}
0&=&\iint_{\rm I(\ell)}\big(\Phi(q,p)-(q,p)\big)\, {\rm d}q \ {\rm d}p\ \cQ({\rm d}\o) \nonumber \\
&=&\iint_{\rm I(\ell)}\big(Q(q,p)-q,P(q,p)-p\big)\ {\rm d}q \ {\rm d}p\ \cQ({\rm d}\o)  \nonumber \\
&=&\iint_{\rm I'(\ell)}\big(Q-\hat q(Q,p,\o),\hat P(Q,p,\o)-p\big)\ {\rm d}q \ {\rm d}p\ \cQ({\rm d}\o) \nonumber \\
&=&\iint_{\rm I'(\ell)}\big(Q-\hat q(Q,p,\o),\hat P(Q,p,\o)-p\big)\ \det \big(\hat q_Q(Q,p,\o)\big) \ {\rm d}Q \ {\rm d}p\ \cQ({\rm d}\o) \nonumber \\
&=&\iint_{\rm I'(\ell)}\big(-b-w_p(Q,p,\o),c+w_Q(Q,p,\o)\big)\ 
\det \big({\rm I}_d+w_{Qp}(Q,p,\o)\big) \ {\rm d}Q \ {\rm d}p\ \cQ({\rm d}\o) \nonumber \\
&=& |{\rm I}(\ell)|(-b,c)-\iint_{{\rm I}'(\ell)}J\nabla w(Q,p,\o)\ \det \big({\rm I}_d+w_{Qp}(Q,p,\o)\big) \ {\rm d}Q\ {\rm d}p\ \cQ({\rm d}\o), \nonumber
\end{eqnarray}
where ${\rm I}_d$ is the $d\x d$ identity matrix,
$\ell=(\ell_1,\dots,\ell_{2d})$, ${\rm I}(\ell)=\prod_{i=1}^{2d}
[-\ell_i,\ell_i]^{2d}$, and
\[
{\rm I}'(\ell)=\big\{(Q,p):\  (\hat q(Q,p,\o),p)\in \rm I(\ell)\big\}.
\]
In summary,
\be\la{eq3.25}
(-b,c)=\frac1{|\rm I(\ell)|}\iint_{{\rm I}'(\ell)}J\nabla w(Q,p,\o)\ \det \big({\rm I_d}+w_{Qp}(Q,p,\o)\big) \ {\rm d}Q\ {\rm d}p\ \cQ({\rm d}\o).
\ee

 Let us write
\begin{align}\la{eq3.26}
&c_0=c_0(\o):=\|\o\|_{{\rm C}^0},\ \ \ \ 
c_1=c_1(\o):=\|{\rm D}\o\|_{\rm C^0},\\ 
\nonumber
&Z_\ell=\int_{{\rm I}(\ell)}\nabla w(Q,p,\o)\ 
\det \big({\rm I}_d+w_{Qp}(Q,p,\o)\big) \ {\rm d} Q \ {\rm d}p,\\
\nonumber
&Z'_\ell=\int_{{\rm I'}(\ell)}\nabla w(Q,p,\o)\ 
\det \big({\rm I}_d+w_{Qp}(Q,p,\o)\big) \ {\rm d} Q \ {\rm d}p.
\end{align}
Observe 
\[
 \rm I(\ell-c_0(\o))\subset \rm I'(\ell)\subset \rm I(\ell+c_0(\o)).
\]
Since
\[
\big|\rm I(\ell+c_0(\o))\setminus \rm I(\ell-c_0(\o))\big|\leq c_2 c_0(\o)\ell^{2d-1},
\]
for a constant $c_2$, we learn
\be\la{eq3.27}
\big|Z_\ell-Z'_\ell\big|
\leq  c_2c_0(\o)^2c_1(\o)^{d}d!\ \ell^{2d-1}.
\ee
Recall that we wish to show $b=c=0$. On account of \eqref{eq3.27}, 
\eqref{eq3.26}, and $\cQ$-almost sure finiteness of $c_0(\o)+c_1(\o)$, it suffices to show 
\be\la{eq3.28}
\lim_{r\to\i}|{\rm I}(\ell^r)|^{-1}\int_{{\rm I}(\ell^r)}\nabla w(Q,p,\o)\ 
\det \big({\rm I}_d+w_{Qp}(Q,p,\o)\big) \ {\rm d} Q \ {\rm d}p=0, 
\ee  
 for $\ell^r$ as in Proposition~\ref{pro3.3}. 

\ms\noi
{\em(Step 3)}
Note that a interchanging $Q$ with $p$, or performing a permutation among
 the variables $(Q_1,\dots, Q_d)$, or $(p_1,\dots,p_d)$
does not alter the integral in \eqref{eq3.28}. 
Because of this, \eqref{eq3.28} would follow if we can show
\be\la{eq3.29}
\lim_{r\to\i}|{\rm I}(\ell^r)|^{-1}\iint_{{\rm I}(\ell^r)}w_{Q_1}(Q,p,\o)\ 
\det \big({\rm I}_d+w_{Qp}(Q,p,\o)\big) \ {\rm d}Q \ {\rm d}p\ \cQ({\rm d}\o)=0.
\ee
To simplify the notation, set
$$w_{Qp}=:A=[a_{ij}]_{i,j=1}^{d}.$$
Expanding the determinant in \eqref{eq3.29} yields
\[
\det( {\rm I}+A)=1+\sum_{k=1}^{d}\sum_{1\leq i_1<\dots<i_k\leq d}
\det[a_{i_ji_l}]_{j,l=1}^k.
\]
This expansion yields an analogous expansion for the left-hand of 
\eqref{eq3.29}. As we examine this expansion, we encounter two types of terms:
Given $k$, either $1\in\{i_1,\dots,i_k\}$, or $1\notin\{i_1,\dots,i_k\}$.
If the former occurs, we perform a permutation to rewrite the corresponding integral as
\[
Z_\ell(\o):=\iint_{{\rm I}(\ell)}w_{Q_1}(Q,p,\o)\ \det \big[w_{Q_ip_j}(Q,p,\o)\big)\big]_{i,j=1}^k \ {\rm d}Q \ {\rm d}p.
\]
If the latter occurs, then we must have $k<d$, and after a permutation, we
 rewrite the corresponding integral as
\[
Z'_\ell(\o):=\int_{{\rm I}(\ell)}w_{Q_1}(Q,p,\o)\ \det \big[w_{Q_{i+1}p_j}(Q,p,\o)\big)\big]_{i,j=1}^{k} \ {\rm d}Q \ {\rm d}p.
\]
We verify\eqref{eq3.29} by showing
\begin{align}\la{eq3.30}
&\lim_{\ell\to\i}|{\rm I}(\ell)|^{-1} Z_\ell(\o)=0,\\
&\lim_{r\to\i}|{\rm I}(\ell^r)|^{-1} Z'_{\ell^r}(\o)=0,\la{eq3.31}
\end{align}
$\cQ$-almost surely. (Recall $\ell^r=\ell^r(\o)$ of Proposition~\ref{pro3.3}
can depend on $\o$.)

\ms\noi
{\em(Step 4)}
We first focus on $Z_\ell$. Let us write
$\bar p=(p_1,\dots,p_{k})$. 
If we fix $Q$ and $$p'=(p_{k+1},\dots,p_{d}),$$ and write 
\[
F(\bar p;Q,p')=F(\bar p) :=\big(w_{Q_1}(Q,p),\dots,w_{Q_k}(Q,p)\big),
\]
then for each $(Q,p')$, the
 ${\rm d}\bar p$ integration in $Z_\ell$ takes the form
\begin{eqnarray}
&& \int_{\bar {\rm I}(\ell)}w_{Q_1}(Q,\bar p,p')\ \det {\rm D}_{\bar p}F(\bar p) \ {\rm d}\bar p \nonumber \\
 &=&\int_{\bar {\rm I}(\ell)}\ \det {\rm D}_{\bar p}F'(\bar p) \ {\rm d}\bar p \nonumber \\
 &=&\int_{F'(\bar {\rm I}(\ell))}\  {\rm d} p_1\wedge \dots\wedge {\rm d}p_k, \nonumber 
\end{eqnarray}
where $D_{\bar p}$ denotes the differentiation with respect of $\bar p$, and
$$\bar {\rm I}(\ell)=\prod_{i=1}^k[-\ell_i,\ell_i],\ \ \ \ 
F'=\big(w_{Q_1}^2/2, w_{Q_2},\dots,w_{Q_k}\big).$$
Here we are using the fact that ${\rm D}_{\bar p}F'$ is obtained from
${\rm D}_{\bar p}F$ by multiplying its first row by $w_{Q_1}$.
Since $\|\nabla w\|_{{\rm C}^0}\leq c_0,$ for
 $c_0=c_0(\o)$ as in \eqref{eq3.26}, we can write
\[
F'\big(\bar {\rm I}(\ell)\big)\subseteq \big[-c^2_0/2,c^2_0/2\big]
\x [-c_0,c_0]^{k-1}.
\]
Hence
\[
\left|\int_{\bar {\rm I}(\ell)}\ \det {\rm D}_{\bar p}F'(\bar p) \ {\rm d}\bar p\right|
\leq 2^{k-1}c_0(\o)^{k+1},
\]
which yields the bound
\[
|Z_\ell(\o)|\leq 2^{k-1}c_0(\o)^{k+1}\prod_{i=k+1}^{2d}(2\ell_i).
\]
This certainly yields  \eqref{eq3.30} because $k\geq 1$, and $c_0(\o)<\i,$
$\cQ$-almost surely.

\ms\noi
{\em(Step 5)}
We now turn our attention to $Z'_\ell$. 
To ease the notation, let us set 
$$f=w_{\hat p}=\big(w_{p_2},\dots,w_{k+1}\big),$$ 
with $$\hat p=(p_2,\dots,p_{k+1}),$$
and regard $f$ as a column vector. With this interpretation, we can write
\[
Z'_\ell=\int_{{I}(\ell)}w_{Q_1}\ \det \big[f_{Q_2},\dots,f_{Q_{k+1}}\big] \ {\rm d}Q\ {\rm d}p.
\]

We wish to integrate by parts with respect to $Q_1$. This can be
performed with a boundary contribution that involves
the functions $w$, and the first derivatives of $f$. More precisely,
\be\la{eq3.32}
Z'_\ell=\sum_{j=2}^{k+1}A^j_\ell+E^1_\ell,
\ee
 where $E^1_\ell$ satisfies a bound of the form 
\be\la{eq3.33}
\big|E_\ell^1\big|\leq  k!\ c_1(\o)^k\int_{\p {\rm I}(\ell)}|w|\ {\rm d}x,
\ee
 and
\[
 A^j_\ell=-\int_{{\rm I}(\ell)}w\ \det \big[f_{Q_2},\dots,f_{Q_{j-1}},
 f_{Q_{j}Q_1},f_{Q_{j+1}},\dots,f_{Q_{k+1}}\big] \ {\rm d} Q\ {\rm d}p,
\]
when $3<j<k$ (and a similar expression when $j\in\{2,3,k,k+1\}$).
Here we are using the fact that $f\in {\rm C}^2$
(because $\Phi\in {\rm C}^2$), and 
\[
\|{\rm D} f\|_{{\rm C}^0}\leq \|{\rm D}^2 w\|_{{\rm C}^0}=\|{\rm D}\o\|_{{\rm C}^0}=c_1(\o).
\]
We now perform an integration by parts with respect to the variable
 $Q_j$. This involves a boundary contribution that again involves $w$ and first derivatives of
$f$. Hence
\be\la{eq3.34}
A^j_\ell=B^j_\ell +\sum_{i=2,i\neq j}^{k+1}C^{ij}_\ell+E^{2,j}_\ell,
\ee
where $E_\ell^{2,j}$ satisfies a bound of the form 
\be\la{eq3.35}
\big|E_\ell^{2,j}\big|\leq  k!\ c_1(\o)^k\int_{\p {\rm I}(\ell)}|w|\ {\rm d}x,
\ee
 and 
\begin{align*}
B^j_\ell&=\int_{{\rm I}(\ell)}w_{Q_j}\ \det \big[f_{Q_2},\dots,f_{Q_{j-1}},
 f_{Q_1},f_{Q_{j+1}},\dots,f_{Q_{k+1}}\big] \ {\rm d} Q \ {\rm d}p,\\
C^{ij}_\ell&=\int_{{\rm I}(\ell)}w\ \det \big[g_2^{ij},\dots,g_{k+1}^{ij}\big] \ {\rm d}Q\ {\rm d}p .
\end{align*}
Here $g^{ij}_r=f_{Q_r}$ when $r\neq i,j$, $g^{ij}_j=
f_{Q_1}$, and $g^{ij}_i=f_{Q_iQ_j}$. From anti-symmetry of the determinant, we can readily see that $$C^{ij}_\ell=-C^{ji}_\ell.$$ 
From this, and \eqref{eq3.32}--\eqref{eq3.35} we deduce,
\be\la{eq3.36}
Z'_\ell=\sum_{j=2}^{k+1} B^{j}_\ell+E^3_\ell,
\ee
with $E^3_\ell$ satisfying  a bound of the form 
\be\la{eq3.37}
\big|E_\ell^3\big|\leq [1+k!k]\ c_1(\o)^k\int_{\p I(\ell)}|w|\ {\rm d}x.
\ee

\ms\noi
{\em(Step 6)} We now study $B^j_\ell$. Let us fix $Q$ and
$$p''=(p_1,p_{k+2},\dots,p_d),$$
and focus on the $\hat p$-integration. We also set $h_r=w_{Q_r}$, and 
$$\hat {\rm I}(\ell)=\prod_{i=2}^{k+1}[-\ell_i,\ell_i].$$ Note that
the expression
\[
B^j_\ell(Q,p''):=\int_{\hat {\rm I}(\ell)}w_{Q_j}\ \det \big[f_{Q_2},\dots,f_{Q_{j-1}},
 f_{Q_1},f_{Q_{j+1}},\dots,f_{Q_{k+1}}\big] \ {\rm d}\hat p,
\]
can be written as
\[
\int_{\hat {\rm I}(\ell)}\ {\rm d}h_2\wedge\dots\wedge {\rm d}h_{j-1}\wedge(h_j\  {\rm d}h_1)\wedge
{\rm d}h_{j+1}\wedge\dots\wedge {\rm d}h_{k+1}.
\]
If we set
\begin{align*}
\hat B^j_\ell(Q,p''):=&\int_{\hat {\rm I}(\ell)}w_{Q_1}\ \det \big[f_{Q_2},\dots,f_{Q_{j-1}},
 f_{Q_j},f_{Q_{j+1}},\dots,f_{Q_k}\big] \ {\rm d}\hat p\\
=&\int_{\hat {\rm I}(\ell)}\ {\rm d}h_2\wedge\dots\wedge {\rm d}h_{j-1}\wedge(h_1\  {\rm d}h_j)\wedge
{\rm d}h_{j+1}\wedge\dots\wedge {\rm d}h_{k+1},
\end{align*}
then  
\[
B^j_\ell(Q,p'')+\hat B^j_\ell(Q,p'')=\int_{\hat {\rm I}(\ell)}\ {\rm d}h_2\wedge\dots\wedge {\rm d}h_{j-1}\wedge {\rm d}(h_jh_1)\wedge
{\rm d}h_{j+1}\wedge\dots\wedge {\rm d}h_{k+1}.
\]
But if $$G(\hat p)=(h_2,\dots, h_{j-1},h_jh_1,h_{j+1},\dots, h_{k+1}),$$ then
\be\la{eq3.38}
B^j_\ell(Q,p'')+\hat B^j_\ell(Q,p'')=\int_{G(\hat {\rm I}(\ell))}\ {\rm d}p_2\wedge\dots\wedge {\rm d}p_{k+1}.
\ee
Since the function $|h_j|$ is bounded by $c_0$, we deduce that 
$$G(\hat {\rm I}(\ell))\subset [-c_0,c_0]^{j-2}\x[-c_0^2,c_0^2]\x [-c_0,c_0]^{k-j+1}.$$ This  and \eqref{eq3.38} imply
\[
\big|B^j_\ell(Q,p'')+\hat B^j_\ell(Q,p'')\big|\leq 2^kc_0(\o)^{k+1}.
\]
From this  we learn 
\[
B^j_\ell=-Z'_\ell+E^{4,j}_\ell,
\] 
where
\[
\big|E_\ell^{4,j}\big|\leq 2^kc_0(\o)^{k+1}(2\ell)^{2d-k}.
\]
This and \eqref{eq3.36} imply, 
\[
Z'_\ell=-kZ'_\ell+E^3_\ell+E^4_\ell, 
\]
with
\be\la{eq3.39}
\big|E_\ell^4\big|\leq k2^kc_0(\o)^{k+1}(2\ell)^{2d-k}.
\ee
In other words, 
\be\la{eq3.40}
(k +1)Z'_\ell=E^3_\ell+E^4_\ell.
\ee
We now divide both sides of \eqref{eq3.40} by $|{\rm I}(\ell)|$, 
and choose $$\ell= \ell^r$$ 
with $\ell^r$ as in Proposition~\ref{pro3.3}, where $v$ is chosen to be  $w$. Finally we send $r\to\i$ and use \eqref{eq3.37}, \eqref{eq3.39}, and Proposition~\ref{pro3.3} to deduce \eqref{eq3.31}. This completes the proof.

\section{The density of fixed points} \label{cpsp}

From Theorem~\ref{th3.3}, we learn that a stationary symplectic twist 
diffeomorphism can be represented as
\[
\Phi\big(Q+w_p(Q,p),p\big)=\big(Q,p+w_Q(Q,p)\big),
\]
for a stationary process $w(Q,p,\o)$. From this representation it is clear
that $$\Phi(Q,p)=(Q,p)$$ if and only if $$f(Q,p,\o):=\widehat\nabla w(Q,p,\o)=0,$$
(or equivalently $\nabla w(Q,p,\o)=0$).
In words, there exists a one-to-one correspondence between the fixed points
of $\Phi$ and the critical points of $w$. We now have the question of existence
of critical points of $w$ before us. To ease the notation, we write $x$ for $(Q,p)$.

We wish to use Ergodic Theorem 
to count the number of points in the zero set of the stationary process $f$,
restricted to a large box.

As a preparation for the statement of the main results of this section, we make some definitions.

\begin{definition}
Recall that a function $f:\bR^n\x\O\to\bR^n$ is {\em $\th$-stationary} if
\[
f(x+a,\o)=f(x,\th_a\o),
\]
or equivalently, $f(x,\o)=\hat f(\th_x\o)$ for $\hat f(\o)=f(0,\o)$.
Given a stationary function $f(x,\o)$, and an open set $U$ of $\bR^{n\x n}$, we define
\[
\cN_U(a,\L ,\o)=\sharp Z_U(a,\L ,\o),
\]
where $Z_U(a,\L ,\o)=Z_U(a,\o)\cap \L$, where
\[
Z_U(a,\o)=\big\{x\in \bR^n :\ f(x,\o)=a,\ {\rm D}f(x,\o)\in U\big\}.
\]
We simply write 
\[
\cN_U(a,\o):=\cN_{U}(a,[0,1]^n ,\o).
\]
As an example of $U$, we may consider the set of matrices $\G \in\bR^{n\x n}$
with exactly $k$ many negative eigenvalues.
\end{definition}

\bp\label{prop4.1} Assume $f$ is $\th$-stationary with respect to the probability measure $\cQ$. If 
\be\la{eq4.1}
\bE \cN_U(a,\o)=\int_\O \cN_U(a,\o)\ \cQ({\rm d}\o)<\i,
\ee
then 
\begin{equation}\label{eq4.2}
\lim_{\ell\to\i}
(2\ell)^{-n}\cN_U\big(a,[-\ell,\ell)^n,\o\big)= \bE \cN_U(a,\o),
\end{equation}
$\cQ$-almost surely, and in $\rm L^1(\cQ)$ sense. 
\ep

\begin{proof}
Observe that if $k$ and $k'$ are two positive integers
with $k<k'$, then
\[
\cN_U\big(a,[k,k')^n,\o\big)=\sum_{i_1,\dots,i_n=k}^{k'-1}
\cN_U\big(a,[i_1,i_1+1)\x\dots\x[i_n,i_n+1),\o\big).
\]
This and the stationarity imply
\be\la{eq4.3}
\bE\ \cN_U\big(a,[k,k')^n,\o\big)=(k'-k)^n\bE \cN_U(a,\o).
\ee
If we assume \eqref{eq4.1}, then the right-hand side of \eqref{eq4.3} is finite.
This in particular implies that the set $Z_U(a,\o)$ is discrete, and that 
$\cN_U\big(a, [-\ell,\ell]^n,\o\big)$ is finite, $\cQ$-almost surely.

By the Ergodic Theorem,
\be\la{eq4.4}
\lim_{\ell\to\i}(2\ell)^{-n}\int_{[-\ell,\ell]^n}\cN_U(a,\th_x\o)\ {\rm d}x=
\bE \cN_U(a,\o),
\ee
almost surely, and in ${\rm L}^1(\cQ)$ sense. We wish to use \eqref{eq4.4}
to deduce  \eqref{eq4.2}. 

Observe that by the stationarity,
\be\la{eq4.5}
Z_U(a, \th_x\o)=Z_U(a,\o)-x:=\big\{y-x:\ y\in Z_U(a,\o)\big\}.
\ee
From this we can readily deduce
\[
\cN_U(a,\L,\th_x\o)=\cN_U(a,\L+x,\o).
\]
This in turn implies
\begin{align*}
\cN_U(a,\th_x\o)&=\cN_U\big(a,[0,1]^n+x,\o\big)
=\sum_{z\in Z(a,\o)}
1\!\!1\left(z\in [0,1]^n+x\right)\\
&=\sum_{z\in Z(a,\o)}
1\!\!1\left(x\in [-1,0]^n+z\right).
\end{align*}
(By convention, the right-hand side is $0$, when $Z_U(a,\o)=\varnothing$.)
 As a consequence, 
\begin{align} \nonumber
\int_{[-\ell-1,\ell+1]^n}\cN_U(a,\th_x\o)\ {\rm d}x&=
\sum_{z\in Z(a,\o)}\int_{[-\ell-1,\ell+1]^n}
1\!\!1\left(x\in [-1,0]^n+z\right)\ {\rm d}x\\
&\geq \cN_U\big(a,[-\ell,\ell]^n,\o\big).\la{eq4.6}
\end{align}
In the same manner we can show
\[
\int_{[-\ell+1,\ell-1]^n}\cN_U(a,\th_x\o)\ {\rm d}x\le
\cN_U\big(a,[-\ell,\ell]^n,\o\big)\ {\rm d}x.
\]
From this, \eqref{eq4.6}, and \eqref{eq4.4}, we can readily deduce \eqref{eq4.2}.
\end{proof}

\bs
On account of \eqref{eq4.3}, we wish to find an explicit formula $\bE\cN(0,\o)$.
In particular we would like to find conditions that would guarantee 
$$\bE\cN(0,\o)>0,$$
so that we can deduce the existence of critical points of $w$. 
Formally speaking, we expect
\be\la{eq4.7}
\bE \ \cN_U(0,\o)=\bE \big[1\!\!1\big(\pmb D \hat f(\o)\in U\big)\ 
|\det \pmb  D\hat f(\o)|\ \d_0\big(\hat f(\o)\big)\big],
\ee
provided that the right-hand side is well-defined. 
(See Definition~\ref{def3.1}(i), for the definition of the operator $\pmb D$.)
In this section, we give one set
of sufficient conditions that would allow us to make sense of \eqref{eq4.7}
(see \eqref{eq4.9} below).
In Section 5, we will be able to use the classical coarea formula to rewrite
the right-hand of \eqref{eq4.7} in a more tractable form when $f$ is quasiperiodic.

As it turns out, 
a multi-dimensional generalization of the classical Kac-Rice formula,
would allow us to express  $\bE\cN(0,\o)$ in terms of the probability 
density of the random variable 
$$
(\hat f(\o),{\pmb D}\hat f(\o)).
$$
We refer to references \cite{AT07,AT09,AW} for  thorough discussions
of Kac-Rice type formulas, and their applications for Gaussian processes.
For our purposes, we need the following variant of Kac-Rice formula. 

\bth\la{th4.2}
Suppose that $f:\bR^n\to\bR^n,\ f(x,\o)=\hat f(\th_x\o)$ is ${\rm C}^2$ $\th$-stationary process.
Assume 
\be\la{eq4.8}
 \bE\ \cN(a,\o)<\i,
\ee
for $a$ near $0$, and  the random variable $(\hat f(\o),{\pmb D}\hat f(\o))$ 
has a density $p(a,\G )$ such that the following conditions are true:
\bi
\item The function 
\[
Q(a):=\int_U |\det \G |\ p(a,\G )\ {\rm d}\G ,
\]
is continuous near $0$. 
\item The function
\[
p(a):=\int_U \ p(a,\G )\ {\rm d}\G,
\]
is bounded near $0$.
\ei
Then
\be\la{eq4.9}
\bE\ \cN_U(0,\o)=\int_U |\det \G | \ p(0,\G )\ {\rm d}\G .
\ee\et

\begin{proof} {\em(Step 1)} According to Area (or Coarea) Formula,
\be \la{eq4.10}
\int_{\bR^n} \z(a)\cN_U(a,\o)\ {\rm d}a=\int_{[0,1]^n}\z(f(x))\ 1\!\!1\big({\rm D}f(x,\o)\in U\big)\ 
|\det {\rm D} f(x,\o)|\ {\rm d}x,
\ee
for every continuous function $\z$ (see for example \cite{AW}). From taking $\cQ$-expectation of both sides  we deduce
\[
\int_{\bR^n} \z(a)\big[\bE \ \cN_U(a,\o)\big]\ {\rm d}a=
\int_{\bR^n}\int_U\z(a) 
|\det \G |\ p(a,\G )\ {\rm d}\G  \ {\rm d}a=\int_{\bR^n}\z(a) Q(a)
\ {\rm d}a.
\]
Hence, 
\be\la{eq4.11}
\bE \ \cN_U(a,\o)=\int_U|\det \G |\ p(a,\G )\ {\rm d}\G ,
\ee
for Lebesgue almost all $a\in\bR^n$. We wish to show that \eqref{eq4.11} holds
for {\em all} $a$ near $0$. To achieve this we will show that for each
$\d>0$, there exists a measurable set $\O(\d)\subset\O$ such that
the following statements hold:
\bi
\item [(i)] $\lim_{\d\to 0}\cQ\big(\O(\d)\big)=1$.
\item [(ii)] If $\o\in\O(\d)$, and $|u|\leq\d$, then $$\cN_U(u,\o)=\cN_U(0,\o).$$
\item [(iii)] If $\o\in\O(\d),$ and $\e\in(0,\d)$, then
\be
 \cN_U(0,\o)=\frac 1{|{\rm B}_{\e}(0)|}
\int_{[0,1]^n}1\!\!1(|f(x,\o)|\leq  \e)\ 1\!\!1\big({\rm D}f(x,\o)\in U\big)\ 
|\det {\rm D}f(x,\o)|\ {\rm d}x.\la{eq4.12}
\ee\ei
Let us demonstrate how (i)-(iii), and \eqref{eq4.11} imply \eqref{eq4.9}.
Indeed from (iii) we deduce
that when $\e\in(0,\d)$, the expression
\[
\bE\ \cN_U(0,\o)1\!\!1\big(\o\in \O(\d)\big),
\]
equals
\begin{align*}
\bE \ &\frac 1{|{\rm B}_{\e}(0)|}
\int_{[0,1]^n}1\!\!1\big(|f(x,\o)|\leq  \e,\ {\rm D}f(x,\o)\in U\big)\ 
|\det {\rm D}f(x,\o)|\ {\rm d}x\ 1\!\!1\big(\o\in \O(\d)\big)\\
&\leq\bE \ \frac 1{|{\rm B}_{\e}(0)|}
\int_{[0,1]^n}\ 1\!\!1\big(|f(x,\o)|\leq  \e,\ {\rm D}f(x,\o)\in U\big)\  |\det {\rm D}f(x,\o)|\ {\rm d}x\\
&=\frac 1{|{\rm B}_{\e}(0)|}\bE \  1\!\!1\big(|f(0,\o)|\leq  \e,\ {\rm D}f(0,\o)\in U\big)\ 
|\det {\rm D}f(0,\o)|\\
&=\frac 1{|{\rm B}_{\e}(0)|}\int_U \int_{{\rm B}_\e(0)}|\det \G | \
p(a,\G )\ {\rm d}a\ {\rm d}\G =
\frac 1{|{\rm B}_{\e}(0)|}\int_{{\rm B}_\e(0)}Q(a)\ {\rm d}a,
\end{align*}
where we use the stationarity for the first equality. 

We then send $\e\to 0$,
and $\d\to 0$ ( in this order), and use the continuity of $Q$ at $0$ to deduce
\be\la{eq4.13}
\bE\ \cN_U(0,\o)\leq  \int_U |\det \G | p(0,\G )\ {\rm d}a\ {\rm d}\G .
\ee
On the other-hand, by \eqref{eq4.11}, we can find a sequence $a_k\to 0$
such that
\be\la{eq4.14}
\bE\ \cN_U(a_k,\o)= \int_U |\det \G | p(a_k,\G )\ {\rm d} a\ {\rm d}\G .
\ee
We use such a sequence to argue
\begin{eqnarray}
\bE\ \cN_U(0,\o)&\leq&  \int_U |\det \G | p(0,\G )\ {\rm d}a\ {\rm d}\G  \nonumber \\
&=&\lim_{k\to\i}\int_U |\det \G | p(a_k,\G )\ {\rm d}a\ {\rm d}\G  \nonumber \\
&=&\lim_{k\to\i}\bE\ \cN_U(a_k,\o) \nonumber \\
&=&\lim_{k\to\i}\lim_{\d\to 0} \bE\ \cN_U(a_k,\o)\ 1\!\!1\big(\o\in \O(\d)\big) \nonumber \\
&=&\lim_{k\to\i}\lim_{\d\to 0}
\bE\ \cN_U(0,\o)\ 1\!\!1\big(\o\in \O(\d)\big) \nonumber \\
&=& \bE\ \cN_U(0,\o), \nonumber
\end{eqnarray}
where we used 
\bi
\item \eqref{eq4.13} for the beginning inequality,
\item the continuity of $Q$ at $0$, for the first equality,
\item \eqref{eq4.14} for the second equality,
\item (i) for the third equality,
\item (ii) for the fourth equality
\item (i) for the last equality. 
\ei
Since we must
have equality for the inequalities in the above display, we arrive at
\eqref{eq4.9}.
 It remains to construct the sets $\O(\d),\ \d>0$, satisfying (i)-(iii).

\ms\noi
{\em(Step 2)} Let us write $Z=Z(\o)$ for the level set $Z_U(a,\o)$, when $a=0$. We also set
\begin{align*}
\O_0&=\big\{\o:\ {\text{there exists}} \ x\in [0,1]^n\ {\text{ such that }}\ 
f(x,\o)=0,\ \det {\rm D}f(x,\o)= 0\big\},\\
\O_1&=\big\{\o:\ {\text{there exists}} \ x\in \p \big([0,1]^n\big)\ {\text{ such that }}\ 
f(x,\o)=0\big\}.
\end{align*}
We assert,
\begin{align}\la{eq4.15}
\cQ(\O_0)&=0,\\
\cQ(\O_1)&=0.\la{eq4.16}
\end{align}
Our assumptions on $f$ would allow us to use Proposition~6.5 of \cite{AW} to 
deduce \eqref{eq4.15}.

Recall that by our assumption $\bE\cN(0,\o)<\i$,
the set $$Z(\o)\cap[0,1]^n$$ is finite almost surely. This and stationarity imply
that $$Z(\o)\cap [-\ell,\ell]^n$$ is finite for every positive $\ell$. Hence the set
$Z(\o)$ is discrete almost surely.  

We now argue that \eqref{eq4.16} 
is a consequence of the discreteness and the stationarity of the set $Z(\o)$.
To explain this, let us write $\l({\rm d}x)$ for the Lebesgue measure on $\bR$, and 
let us write $$\pi_i(x_1,\dots,x_n)=x_i$$ for the $i$-th coordinate projection.
Evidently, the discreteness of $Z$ implies that $\l(Z(\o))=0$, which in turn implies 
\[
(\l\x \cQ)\big(\big\{(a,\o):\ a\in \pi_iZ(\o)\big\}\big)=: (\l\x \cQ)(\cZ_i)=0,
\]
by Fubini's theorem. From this and Fubini's theorem again we learn
\be\la{eq4.17}
\l\left(\left\{a: \ \cQ\big(\big\{\o:\ (a,\o)\in \cZ_i\big\}\big)>0\right\}\right)=0.
\ee
By stationarity of $Z(\o)$, the probability
\[
\cQ\big(\big\{\o:\ (a,\o)\in \cZ_i\big\}\big),
\]
is independent of $a$. From this and \eqref{eq4.17} we deduce
\[
\cQ\big(\big\{\o:\ (a,\o)\in \cZ_i\big\}\big)=0,
\]
for every $a\in \bR$. In particular,
\[
\cQ\big(\big\{\o:\ \pi_iZ(\o)\cap\{0,1\}\neq \varnothing\big\}\big)=0,
\]
As a consequence,
\[
\cQ\big(\big\{\o:\ \pi_iZ(\o)\cap\{0,1\}\neq \varnothing\ {\text{ for some }}\ i\in\{1,\dots,n\}\big\}\big)=0.
\]
This is exactly the claim \eqref{eq4.16}.
 
\ms\noi
{\em(Step 3)} We are now ready to construct our sets $\O(\d),\ \d>0$.
Take $\o\in\O\setminus\big(\O_0\cup\O_1)$, and assume that
$N=N(\o)=\sharp Z(\o)\neq 0$. If 
\[
Z(\o)=\big\{a^1,\dots, a^N\big\},
\]
then we can use the (local) inverse mapping theorem to
find $\e(\o)>0$, and disjoint open sets 
\[
V_1(\o),\dots,V_N(\o)\subset [0,1]^n\cap U,
\]
 such that $a^i\in V_i(\o)$, $f(V_i(\o))={\rm B}_{\d(\o)}(0)$, and the restriction of
$f$ to each $V_i$ is a diffeomorphism for each $i$.
To have $\e(\o)$ a measurable function,
we choose $\e=\e(\o)$ to be the largest positive number for which such sets $V_1,\dots, V_N$ exist. We set 
\[
\O(\d)=\big\{\o\in \O\setminus\big(\O_0\cup\O_1):\ \e(\o)\geq \d\big\}.
\]
We now verify (i)-(iii) of {\em(Step 1)}.
The property (i) is an immediate consequence of \eqref{eq4.14} and \eqref{eq4.15}.
The property (ii) follows from the fact that $\cN_U(a,\o)=N$ for 
$u\in {\rm B}_{\d(\o)}(0)$.
To verify the third property, set
\[
W_\e(\o)=\big\{x:\ |f(x,\o)|\leq \e\big\}.
\]
When $\e<\d(\o)$, the right-hand side of \eqref{eq4.12} equals
\begin{align*}
\frac 1{|{\rm B}_{\e}(0)|}&
\sum_{i=1}^N\int_{V_i(\o)\cap W_\e(\o)}|\det {\rm D}f(x,\o)|\ {\rm d}x\\
&=\frac 1{|{\rm B}_{\e}(0)|}
\sum_{i=1}^N\big|f(V_i(\o)\cap W_\e(\o))\big|
=N,
\end{align*}
because $f(V_i(\o)\cap W_\e(\o))={\rm B}_{\e}(0)$. This completes the verification
of the third property.
\end{proof}

\ms
\begin{remark}
In Theorem~\ref{th4.2} we assumed that the law of the random variable
 $(\hat f,\pmb D\hat f)$ has a density $p(a,\G)$. This requirement can be replaced with the following two conditions:
\bi
\item The law of the random variable $\hat f$ has a density $p(a)$
that is continuous near $0$.
\item If $q(a,d\G)$ is the the conditional probability distribution of $\pmb D \hat f$,
given $\hat f=a$, then the function
\[
a\mapsto \int_{\bR^{n\x n}} |\det \G|\ q(a,{\rm d}\G),
\]
is continuous at $a=0$. We remark that if $(\hat f,\pmb D\hat f)$ has a density $p(a,\G)$, then $q$ in terms of $p$ is given by $q(a,\G)/p(a) {\rm d}G.$
\ei
\end{remark}

Theorem~\ref{th4.2} and Proposition~\ref{prop4.1} give us a way of counting the fixed points of 
a stationary symplectic twist diffeomorphism
provided that the conditions of Theorem~\ref{th4.2} are met.
 The conditions of this theorem are stated for the density of the pair
$$(f,\rm Df)=\big(\widehat\Phi-{\rm id}\ , \ \rm D\widehat \Phi-I\big).$$
In practice, we need conditions that are formulated for
the original symplectic map $\Phi$, not $\widehat \Phi$.
The following result will remedy this.

\bp\la{pro4.1}
Let $\Phi$ be as in Theorem~\ref{th3.3}. Assume that  
the pair $$(\Phi(x),{\rm D}\Phi(x))$$ has a density $\rho(x,X,\G)$ with respect
to the Lebesgue measure 
${\rm d}X\ {\rm d}\G$ of $\bR^{2d}\x\bR^{2d\x 2d}$.
Then the pair $(\widehat\Phi(x),{\rm D}\widehat\Phi(x))$ has a density 
$\hat\rho(\hat x,\hat X,\hat \G)$ with respect
to the Lebesgue measure ${\rm d}\hat X\ {\rm d}\hat\G$,
where
\be\la{eq4.18}
\hat\rho(\hat x,\hat X,\hat \G)=|\det \hat A|^{1-4d}
\rho(x,X,\cA(\hat\G)),
\ee
with $x=(q,p)$, $X=(Q,P)$, $\hat x=(Q,p)$, $\hat X=(q,P)$, and
\be\la{eq4.19}
\cA(\hat\G):=\begin{bmatrix}\hat A^{-1}&-\hat A^{-1}\hat B
\\ \hat C\hat A^{-1}&\hat D-\hat C\hat A^{-1}\hat B\end{bmatrix},\ \ \ \
{\textup{for }}\ \ \ \ \hat\G=\begin{bmatrix}\hat A&\hat B\\\hat C&\hat D\end{bmatrix}.
\ee
\ep

\begin{proof} {\em(Step 1)} We first find an expression for 
${\rm D}\widehat\Phi$ in terms of
${\rm D} \Phi$. From the definition of $\hat q(Q,p),$ and 
$$\hat P(Q,p)=P(\hat q(Q,p),p), $$ 
we learn
\begin{eqnarray}
\big(Q_q,Q_p\big) &=&\big(\hat q_Q^{-1},-\hat q_Q^{-1}\ \hat q_p\big), \nonumber \\
\big(P_q, P_p\big)&=&\big(\hat P_Q\ \hat q_Q^{-1}, 
 \hat P_p-\hat P_Q\ \hat q_Q^{-1}\ \hat q_p\big). \nonumber
\end{eqnarray}
From $$\widehat \Phi(Q,p)=\big(\hat q,\hat P\big)(Q,p)$$ we learn that if
\begin{align*}
\G:&={\rm D}\Phi=\begin{bmatrix}A&B\\C&D\end{bmatrix}:=\begin{bmatrix}Q_q&Q_p\\ P_q&P_p\end{bmatrix},\\
\hat\G:&=
{\rm D}\hat\Phi=\begin{bmatrix}\hat A&\hat B\\\hat C&\hat D\end{bmatrix}
:=\begin{bmatrix}\hat q_Q&\hat q_p\\ \hat P_Q&\hat P_p\end{bmatrix},
\end{align*}
then $\G=\cA(\hat\G)$, with $\cA(\hat\G)$ given by \eqref{eq4.19}.
Clearly,
\be\la{eq4.20}
\rho(x, X,\G)\ {\rm d}X\ {\rm d}\G=m(\hat x, \hat X,\hat \G)\ {\rm d}\hat X\ {\rm d}\hat\G,
\ee
where 
 $$m(\hat x, \hat X,\hat \G)=m(Q,p, q,P,\hat \G)=
\rho\big(q,p,Q,P,\cA(\hat \G)\big).$$
From 
$${\rm d}X= {\rm d}q\ {\rm d}p =\big|\det (\hat q_Q)\big|\ {\rm d}Q\ {\rm d}p,$$
we deduce
\be\la{eq4.21}
{\rm d}X=\big|\det \hat A\big| \ {\rm d}\hat X.
\ee

It remains  to express ${\rm d}\G$ in terms of ${\rm d}\hat\G$.
We write ${\rm d}A, \ {\rm d}B,\ {\rm d}C,$ and ${\rm d}D$ for the volume forms in
$\bR^{d\x d}$ associated with $A,\ B, \ C$ and $D$. 
Here by ${\rm d}A$ we really mean
\[
{\rm d}A=\big({\rm d}a_{11}\wedge {\rm d}a_{12}\wedge \dots\wedge {\rm d} a_{1d}\big)
\wedge\dots\wedge \big({\rm d}a_{d1}\wedge {\rm d}a_{12}\wedge \dots \wedge
{\rm d}a_{dd}\big),
\]
where $a_{ij}, \ i,j=1,\dots d,$ denote the entries of $A$. In the same manner
we define ${\rm d}B,\ {\rm d}C,$ and ${\rm d}D$.
Analogously
${\rm d}\hat A, \ {\rm d}\hat B, \ {\rm d}\hat C$, and ${\rm d}\hat D$ are defined.
On account of \eqref{eq4.20}, \eqref{eq4.21}, the proof is complete if we show 
\be\la{eq4.22}
({\rm d}A)\wedge ({\rm d}B)\wedge ({\rm d}C)
\wedge ({\rm d}D)=\pm
(\det \hat A)^{-4d}\ ({\rm d}\hat A)\wedge 
({\rm d}\hat B)\wedge ({\rm d}\hat C)\wedge ({\rm d}\hat D).
\ee
Note the equality in \eqref{eq4.22} is for volume forms, where as in \eqref{eq4.20}
the expressions $dX\ d\G$ and $d\hat X\ d\hat\G$ refer to the measures
in $\bR^{2d}\x \bR^{2d\x 2d}$. Since we are interested in the probability densities of our random variables, we do not keep
track of signs of the volume forms that will appear in the subsequent calculations.

\ms\noi
{\em(Step 2)} We first try to express $${\rm d}A={\rm d}(\hat A^{-1})$$ in terms of 
${\rm d}\hat A$. By our Lemma~\ref{lem4.1} below, we have
\be\la{eq4.23}
{\rm d}A=\pm (\det \hat A)^{-2d}\ {\rm d}\hat A.
\ee
We next study 
\[
({\rm d}A)\wedge ({\rm d}B)=\big({\rm d}\hat A^{-1}\big)\wedge {\rm d}\big(-\hat A^{-1}\hat B\big).
\]
As we apply the exterior derivative on $\hat A^{-1}\hat B=A\hat B$, we can 
treat $A^{-1}$ as a constant because of the wedge
product with $\rm d\hat A$ (we are using ${\rm d}\hat a^{ij}\wedge {\rm d}\hat a^{ij}=0$).

Lemma~\ref{lem4.1} below allows us write 
\be\la{eq4.24}
{\rm d}(E \hat B)=\pm(\det E)^d\ {\rm d}\hat B,
\ee
for a constant matrix $E$.  Hence
\begin{align*}
({\rm d}A)\wedge {\rm d}B)&=\pm (\det \hat A)^{-2d}\ ({\rm d}\hat A)\wedge
\big({\rm d}\big(\hat A^{-1}\hat B\big)\big)\\
&=\pm (\det \hat A)^{-2d}\ \left(\det\big(\hat A^{-1}\big)\right)^{d}
({\rm d}\hat A)\wedge
({\rm d}\hat B)\\
&=\pm (\det \hat A)^{-3d}\ ({\rm d}\hat A)\wedge
({\rm d}\hat B).
\end{align*}
In the same fashion,
\[
({\rm d}A)\wedge ({\rm d}B)\wedge ({\rm d}C)=\pm (\det \hat A)^{-4d}\ ({\rm d}\hat A)\wedge
({\rm d}\hat B)\wedge ({\rm d}\hat C).
\]
From this and 
\[
({\rm d}A)\wedge ({\rm d}B)\wedge ({\rm d}C)\wedge ({\rm d}D)=
({\rm d}A)\wedge ({\rm d}B)\wedge ({\rm d}C)\wedge ({\rm d}\hat D),
\]
we can readily derive \eqref{eq4.22}.
\end{proof}

\ms
It remains to verify \eqref{eq4.23} and \eqref{eq4.24}.

\ms
\begin{lemma}\la{lem4.1}
The following statements hold.
\begin{itemize}
\item[(i)] Given $E\in\bR^{d\x d}$, consider the map 
$$\z \colon \bR^{d\x d}\to \bR^{d\x d},$$ such that $\z(Z)=EZ.$
If we write ${\rm d}Z$ for the volume form of $\bR^{d\x d}$, then 
\be\la{eq4.25}
\z^*({\rm d}Z)= \pm(\det E)^{d}\ {\rm d}Z. 
\ee
A similar formula is valid if $\z(Z)=ZE$.
\item[(ii)] If $\eta(Z)=Z^{-1}$, then 
\be\la{eq4.26}\eta^*({\rm d}Z)=\pm(\det Z)^{-2d}\ {\rm d}Z.\ee
\end{itemize}
\end{lemma}

\begin{proof}
(i) Given $$Z=[z_{ij}]_{i,j=1}^d,\,\,\, \, \, \, \, EZ=[\z_{ij}]_{i,j=1}^d,$$
we can write
\begin{eqnarray}
\bigwedge_{j=1}^d\  \bigwedge_{i=1}^d\  {\rm d}\z_{ij}
&=& \bigwedge_{j=1}^d\  \bigwedge_{i=1}^d \left(\sum_{k=1}^d e_{ik}
\ {\rm d}z_{kj}\right) \nonumber \\
&=& \bigwedge_{j=1}^d\  \sum_{\s\in S_d}\prod_{i=1}^d e_{ i\s(i)}\ 
\bigwedge_{j=1}^d {\rm d}z_{\s(i)j}   \nonumber \\
&=&\bigwedge_{j=1}^d\  \sum_{\s\in S_d}\e(\s)\prod_{k=1}^d e_{i \s(i)}\ 
\bigwedge_{i=1}^d {\rm d}z_{ij} \nonumber \\
&=&(\det E)^d\bigwedge_{j=1}^d\  
\bigwedge_{i=1}^d {\rm d}z_{ij} ,\nonumber 
\end{eqnarray}
where $S_d$ denotes the set of permutations of $\{1,\dots,d\}$.
This completes the proof of \eqref{eq4.25}.

\ms\noi
(ii) Let us write $\eta^*({\rm d}Z)=\a(Z)\ {\rm d}Z$. Fix $E$ as in part (i).
Observe
\begin{eqnarray}
(\eta\circ\z)^* ({\rm d}Z)&=&\z^*\eta^*({\rm d}Z) \nonumber \\
&=& \zeta^*\big(\a(Z){\rm d}Z\big) \nonumber \\
&=&\pm (\det E)^d\ \a(EZ)\ {\rm d}Z. \la{eq4.27}
\end{eqnarray}
On the other hand, if we set $\xi(Z)=ZE^{-1}$, then
$\eta\circ \z=\xi\circ \eta$, and 
\begin{eqnarray}
(\xi\circ\eta)^* (dZ)&=&\eta^*\xi^*({\rm d}Z) \nonumber \\
&=&\pm (\det E)^{-d}\ \eta^*({\rm d}Z) \nonumber \\
&=&(\det E)^{-d}\ \a(Z)\ {\rm d}Z. \nonumber
\end{eqnarray}
From this and \eqref{eq4.27} we deduce
\[
\a(EZ)=\pm (\det E)^{-2d}\ \a(Z),
\]
which in turn implies that 
\be\la{eq4.28}
a(E)=(\det E)^{-2d}\a({\rm I}),
\ee
where $\rm I=\rm I_d$ denotes the $d\x d$ identity matrix.
Furthermore, since $\eta\circ\eta={\rm id}$, we know
\begin{eqnarray}
{\rm d}Z&=& (\eta\circ\eta)^*({\rm d}Z) \nonumber \\
&=& \eta^*\big(\a(Z)\ {\rm d}Z\big)  \nonumber \\
&=&\a(Z^{-1})\a^*({\rm d}Z) \nonumber \\
&=&\a(Z^{-1}) 
\a(Z)\ {\rm d}Z, \nonumber 
\end{eqnarray}
which in particular implies that $\a(\rm I)^2=1$. From this and \eqref{eq4.28} we can readily deduce \eqref{eq4.26}.
\end{proof}

\ms
\begin{remark} We express our formula \eqref{eq4.18} in terms of the density 
of $\big(\Phi(x) , \rm D\Phi(x)\big)$ with no reference to the stationarity of the process
$\o=\cF(\Phi)$. In fact, if the law of pair $(\o(x),\rm D\o(x))$ with respect to the measure $\cQ$ has a density $\g(X,\G)$, then it does not depend on $x$ by stationarity, and $\rho$ can be expressed in terms of $\g$ by the following formula:
\[
\rho(x, X, \G)=\g(X-x, \G-\rm I),
\]
where $I$ is $(2d)\x(2d)$ identity matrix.
\end{remark}

\section{The density of fixed points in the case of quasiperiodic maps} \label{qsp}

Theorem~\ref{th3.3} and Proposition~\ref{prop4.1} reduce the counting of the fixed points to the evaluation of $\bE\cN_U(0,\o)$. 

If we take the expected value of both sides
of \eqref{eq4.10}, and use the stationarity we always have
\be\la{eq5.1}
\int \z(a)\big[\bE \cN_U(a,\o)\big]\ da=\int
\z(\hat f(\o))\ 1\!\!1\big(\pmb D\hat f(\o)\in U\big)\ 
|\det \pmb D\hat f(\o)|\ \cQ({\rm d}\o),
\ee
for every bounded continuous function $\z$. (See Definition~\ref{def3.1}(i)
for the definition of $\pmb D$.)

If we can choose $\z$ to be the delta function at $0$, then we have the 
informal expression \eqref{eq4.7} for $\bE\cN_U(0,\o)$. 
Theorem~\ref{th4.2} offers an explicit formula for the right-hand side of \eqref{eq4.7} in terms of the density of the pair 
$(\hat f,{\pmb D}\hat f)).$
The existence of a density is rather a restrictive requirement and not valid for many examples of interest. 
In this section, we offer a new explicit formula for $\bE\cN_U(0,\o)$
when $\Phi$ is quasi periodic.

Recall that $\hat f=\widehat {\pmb \nabla} \hat w$. To simplify our notation, we may instead consider the zero set of the function $\hat g=\pmb \nabla \hat w$.
Before stating our first main result, let us review 
the setting we will be working with.

\ms\noi
{\bf Setting 5.1} Given a ${\rm C}^2$ function $$\hat w:\bT^N\to\bR,$$
and $N\x n$ matrix $A$, define 
 $$w(x)=w(x,\o)=\hat w(\Theta_{Ax} \o),$$ 
$\hat g=\pmb\nabla \hat w$, and  $g=\nabla w$. 
We assume that the $n\x n$ matrix $$E:=A^*A$$ is of full rank.
Let us write 
\[
\cN_U(a,\o)=\sharp\left\{x\in[0,1]^{n}:\ g(x,\o)=a,\
{\rm D}g(x,\o))\in U\right\}.
\]
Observe 
\be\la{eq5.2}
\hat g(\o)=\nabla \hat w(\o)A,\ \ \ \ {\pmb D} \hat g(\o)=A^* {\rm D^2}\hat w(\o)A,
\ee
where $\nabla \hat w $ and ${\rm D} \hat g$ represent the standard derivatives 
of $\hat w$ and $\hat g$ (as opposed to the $\pmb\nabla$ and $\pmb D$ which
denote the differentiation in the sense of Definition~\ref{def3.1}(i)).
We write $\cQ$ for the Lebesgue measure on $\bT^N$.
\qed

\ms
We are now ready to present our formula for $\bE \cN_U(0,\o)$,
namely \eqref{eq5.3} below. The proof of \eqref{eq5.3} is similar to the proof of \eqref{eq4.9}. One of the main tool we use is the celebrated Coarea
Formula that we now recall; given a measurable set $V$, a ${\rm C}^1$ 
function $$S:\bR^N\to\bR^n,$$ with $N>n$,
and a non-negative measurable function $T:\bR^N\to\bR$, we have
\be\la{eq5.3}
\int_V T(x)\ (\cJ S)(x)\ {\rm d}x=\int_{\bR^n}\left[\int_{V\cap S^{-1}(a)} T\ {\rm d}\s_{N-n}
\right]\ {\rm d}a,
\ee
where $$(\cJ S)(x)=\big(\det\big(({\rm D}S)(x)({\rm D}S)^*(x)\big)\big)^{1/2},$$
(here $A^*$ denotes the transpose of $A$), and $\s_{N-n}$ denotes the $N-n$ dimensional (Hausdorff) measure. 
For our purposes, we wish to choose
\[
T(x)=W(x) (\cJ S)(x)^{-1},
\]
in \eqref{eq5.3}. This function is well-defined so long as
\[
V\subset \Sigma:=\big\{X\in\bR^N:\ (\cJ S)(x)\neq 0\big\}.
\]
For such a choice of $T$, \eqref{eq5.3} reads as
\be\la{eq5.4}
\int_V W(x)\ {\rm d}x=\int_{\bR^n}\left[\int_{V\cap S^{-1}(a)} W
(\cJ S)^{-1}\ {\rm d}\s_{N-n}.
\right]\ {\rm d}a,
\ee

\ms
We are now ready to state and prove the main result of this section.

\bth\la{th5.1} Let $\hat w$ and $\hat g$ be as in Setting 5.1. Then
\be\la{eq5.5} 
\bE \cN_U(0,\o)=\int_{\bT^N}\cN_U(0,\o)\ {\rm d}\o=
\int_{\L(0)}\frac{\big|\det(A^*{\rm D}^2\hat w(\o)A)\big|}
{\det(A^*({\rm D}^2\hat w(\o))^2 A)^{1/2}}\ \s_{N-n}({\rm d}\o),
\ee
where,
\[
\L(a)=\big\{\o\in\bT^N:\ \nabla\hat  w(\o)A=a,\ A^*{\rm D}^2\hat w(\o)A\in U,\ \det\big(A^*{\rm D}^2\hat w(\bar\o)A\big)\neq 0\big\}.
\]
Moreover, if  there exists $\bar\o$ such that 
$$ \nabla\hat w(\bar\o)A=0,\ \ \ \ A^*{\rm D}^2\hat w(\bar\o)A\in U, $$
and $$A^*{\rm D}^2\hat w(\bar\o)A$$ is invertible,  then
the right-hand side of \eqref{eq5.4} is nonzero.
\et

\begin{proof} {\em(Step 1)}
We may apply the area formula to the function $g$ to derive 
the analogue of \eqref{eq5.1}, 
\be\la{eq5.6}
\int_{\bR^n} \z(a)\big[\bE \cN_U(a,\o)\big]\ {\rm d}a=\int_{\bT^N}
\z(\hat g(\o))\ 1\!\!1\big(\pmb D\hat g(\o)\in U\big)\ 
|\det \pmb D\hat g(\o)|\ \rm d\o,
\ee
for every continuous function $\z$. We wish apply Coarea Formula \eqref{eq5.4}
to the right-hand side of \eqref{eq5.6}, for the choices of 
\begin{align}\nonumber
&S(\o)= \hat g(\o),\ \ \ \ W(\o)=\z(\hat g(\o))\ ,\\
&V=\big\{\o:\ \pmb D\hat g(\o)\in U\big\}\cap \Sigma,\la{eq5.7}
\end{align}
where
\[
\Sigma=\big\{\o\in\bT^N:\ (\cJ \hat g)(\o)\neq 0\big\}.
\]
Observe , 
\[
{\rm D} \hat g(\o)=A^* {\rm D}^2\hat w(\o),\ \ \ \ (\cJ \hat g)(\o)=\det(A^*({\rm D}^2\hat w(\o))^2 A)^{1/2}.
\]
On the other hand, if $\o\notin\Sigma$, then there exists a nonzero vector 
$b$ such that $A^*C^2Ab=0$, for $C={\rm D}^2 \hat w(\o)$. As a consequence,
$A^*CA b=0$ because
\[
0=A^*C^2Ab\cdot b=|CAb|^2.
\]
From this we learn
\[
\big\{x\in\bT^N:\ \det(A^*{\rm D}^2\hat w(\o)A)\neq 0\big\}\subset \Sigma.
\]
Because of this, the right-hand side of \eqref{eq5.6} equals to
\be\la{eq5.8}
\int_{\Sigma}
\z(\hat g(\o))\ 1\!\!1\big(\pmb D\hat g(\o)\in U\big)\ 
|\det \pmb D\hat g(\o)|\ \rm d\o.
\ee
 We now apply  Coarea Formula \eqref{eq5.4} to \eqref{eq5.8},
for the choices of \eqref{eq5.7}, to assert
\be\la{eq5.9}
\int_{\bR^n} \z(a)\big[\bE \cN_U(a,\o)\big]\ {\rm d}a=\int_{\bR^n}
\z(a)\ \left[ \int_{\L(a)}\frac{\big|\det(A^*{\rm D}^2\hat w(\o)A)\big|}
{\det(A^*({\rm D}^2\hat w(\o))^2 A)^{1/2}}\ \s_{N-n}({\rm d}\o)\right]\ 
{\rm d}a.
\ee
From this we deduce
\be\la{eq5.10}
\bE \cN_U(a,\o)=
 \int_{\L(a)}\frac{\big|\det(A^*{\rm D}^2\hat w(\o)A)\big|}
{\det(A^*({\rm D}^2\hat f(\o))^2 A)^{1/2}}\ \s_{N-n}({\rm d}\o)=:G(a),
\ee
for Lebesgue almost all $a\in\bR^n$. We wish to show that \eqref{eq5.7} holds
for {\em all} $a$. We achieve this by verifying the continuity of the function $G$,
and a repetition of some of the steps of the proof of Theorem~\ref{th4.2}.

\ms\noi
{\em(Step 2)} With a verbatim argument as in the proof of Theorem~\ref{th4.2}, we can show that there exists a collection of measurable sets 
$\big\{\O(\d):\ \d>0\}$, with $\O(\d)\subset\O=\bT^N$, such that 
the following statements hold:
\bi
\item [(i)] $\lim_{\d\to 0}\cQ\big(\O(\d)\big)=1$.
\item [(ii)] If $\o\in\O(\d)$, and $|u|\leq\d$, then $\cN_U(u,\o)=\cN_U(0,\o).$
\item [(iii)] If $\o\in\O(\d),$ and $\e\in(0,\d)$, then
\be
 \cN_U(0,\o)=\frac 1{|{\rm B}_{\e}(0)|}
\int_{[0,1]^n}1\!\!1(|g(x,\o)|\leq  \e)\ 1\!\!1\big({\rm D}g(x,\o)\in U\big)\ 
|\det {\rm D}g(x,\o)|\ {\rm d}x.\la{eq5.11}
\ee\ei
Let us demonstrate how the continuity of $G$, 
 (i)-(iii), and \eqref{eq5.10} for almost all $a$, imply that \eqref{eq5.10}
holds for $a=0$ (continuity at any other $a$ can be shown in exactly the same way).

As in the proof of Theorem~\ref{th4.2} we use \eqref{eq5.11} to assert that 
when $\e\in(0,\d)$,
the expression
\[
\bE\ \cN_U(0,\o)1\!\!1\big(\o\in \O(\d)\big),
\]
equals
\begin{align*}
\bE \ &\frac 1{|{\rm B}_{\e}(0)|}
\int_{[0,1]^n}1\!\!1\big(|g(x,\o)|\leq  \e,\ {\rm D}g(x,\o)\in U\big)\ 
|\det {\rm D}g(x,\o)|\ {\rm d}x\ 1\!\!1\big(\o\in \O(\d)\big)\\
&\leq\bE \ \frac 1{|{\rm B}_{\e}(0)|}
\int_{[0,1]^n}\ 1\!\!1\big(|g(x,\o)|\leq  \e,\ {\rm D}g(x,\o)\in U\big)\  |\det {\rm D}g(x,\o)|\ {\rm d}x\\
&=\frac 1{|{\rm B}_{\e}(0)|}\bE \  1\!\!1\big(|g(0,\o)|\leq  \e,\ {\rm D}g(0,\o)\in U\big)\ 
|\det {\rm D}g(0,\o)|\\
&=\frac 1{|{\rm B}_{\e}(0)|}\int_{\bT^N}   1\!\!1\big(|\hat g(\o)|\leq  \e,\ {\pmb D}
\hat g(\o)\in U\big)\ 
|\det {\pmb D}\hat g(\o)| \ {\rm d}\o\\
&=\frac 1{|{\rm B}_{\e}(0)|}\int_{{\rm B}_\e(0)} G(b)\ \rm db,
\end{align*}
where we used the stationarity for the first equality, and Coarea Formula for the last equality. 

We then send $\e\to 0$,
and $\d\to 0$ ( in this order), and use the continuity of $G$ to deduce
\be\la{eq5.12}
\bE\ \cN_U(0,\o)\leq  G(0) .
\ee
On the other-hand, by the validity of \eqref{eq5.10} for almost all points, we can find a sequence $a_k\to 0$
such that
\be\la{eq5.13}
\bE\ \cN_U(a_k,\o)= G(a_k) .
\ee
As in the proof of Theorem~\ref{th4.2}, we use \eqref{eq5.12}, the continuity of $G$, \eqref{eq5.13}, and  (i)-(ii),
to argue
\begin{eqnarray}
\bE\ \cN_U(0,\o)&\leq&  G(0)=\lim_{k\to\i}G(a_k)  \nonumber \\
&=&\lim_{k\to\i}\bE\ \cN_U(a_k,\o) \nonumber \\
&=&\lim_{k\to\i}\lim_{\d\to 0} \bE\ \cN_U(a_k,\o)\ 1\!\!1\big(\o\in \O(\d)\big) \nonumber \\
&=&\lim_{k\to\i}\lim_{\d\to 0}
\bE\ \cN_U(0,\o)\ 1\!\!1\big(\o\in \O(\d)\big) \nonumber \\
&=& \bE\ \cN_U(0,\o). \nonumber
\end{eqnarray}
Since we must have equality for the inequalities in the above display, 
we arrive at
\[
\bE\ \cN_U(0,\o)=G(0),
\]
 which is \eqref{eq5.5}.
 It remains to verify the continuity of the function $G$.

\ms\noi
{\em(Step 3)} As a preparation for the proof of the continuity, we first study
the level set $\L(a)$, which is a subset of $\Sigma$.
Observe that if $\o\in\Sigma$, then 
\[
\det \big(A^*C^2A\big)=\det (M^*M)\neq 0,
\]
where $C=C(\o)=D^2 \hat w$ as before, and $M=CA$. Note that 
$M\in \bR^{N\x n}$, with $N>n$. Hence, we may apply  Cauchy-Binet Formula,
to write
\be\la{eq5.14}
\det (M^*M)=\sum_{I\in \cI}(\det M_I )^2,
\ee
where $\cI$ denotes  the collection of sets 
$I\subset \{1,2,\dots,N\}=:[N]$ such that $\sharp I=n$,
and for $$M=[m_{ij}]_{i\in[N],j\in[n]},$$ by $M_I$ we mean the $n\x n$ submatrix 
of $M$, given by
\[
M_I=[m_{ij}]_{i\in I,j\in[n]}.
\]
Let us write $$\o=(\o_1,\dots,\o_N)$$ for the coordinates of $\o\in\bT^N$
(regarding $\bT^N=[0,1]^N$, with $0=1$).
We also write 
\[
\nabla_I=\left(\frac{\p}{\p \o_j}:\ j\in I\right),
\]
and 
$\Sigma_I$ for the set of $\o\in \bT^N$ such that  
\begin{align*}
{\nabla}_I\hat g(\o)&=\nabla_I\nabla \hat w A=
\left[\sum_{k=1}^N\hat w_{\o_i\o_k}a_{kj}\right]_{i\in I,j\in [n]}=[M_{ij}]_{i\in I,j\in [n]}=M_I,
\end{align*} 
is invertible. From \eqref{eq5.14} we learn
\be\la{eq5.15}
\Sigma=\cup_{I\in\cI}\Sigma_I.
\ee
We now examine the set $\L(a)\cap \Sigma_I$, for each $I\in\cI$.

Without loss of generality, we may assume that $I=[n]$.
Let us examine the set $\Sigma_{I}$, when $I= [n]$.
Regarding $\o\in\bT^N$, as a point in $[0,1]^N$, 
we may write $\o=(\o^1,\o^2)\in \bR^{n}
\x\bR^{N-n}$. For $\o\in\Sigma_{[n]}$, we know  
$${\nabla}_{\o^1}\hat g(\o)={\p}^2_{\o^1\o}\hat w(\o)A$$ 
is invertible. Fix 
\[
\bar\o=(\bar \o^1,\bar\o^2)\in \Sigma_{[n]}\cap \L(a).
\]
If we define $$F(\o^1,\o^2)=(\hat g(\o),\o^2),$$ then ${\rm D}F(\bar\o)$ is invertible, and for $b$ in a neighborhood of  
\[
\bar b=(\hat g (\bar \o),\bar\o^2)=(a,\bar \o^2),
\]
we can define $F^{-1}(b)$. This, and the compactness of $\L(a)$ allow us to find an open covering 
\[
\Sigma_I\cap \L(a)\subset \bigcup_{i=1}^{\a(I)}\Sigma_{I}^i,
\] 
such that for every $i\in\{1,\dots,\a(I)\}$, 
\[
F(\Sigma_I^i)={\rm B}_{\d_i}(a)\x U_i^I,
\]
for some $\d_i>0$, and some open set $U_i^I\subset \bR^{N-n}$, 
and we can make sense of 
$$F^{-1}(a,\o^2)=(R_{I}^i(a,\o^2),\o^2),$$ with 
\[
R_I^i:{\rm B}_{\d_i}(a)\x U_I^i\to
\bR^{n},
\]
a ${\rm C}^1$-function. In other words,
$$F\big(R_I^i(a,\o^2),\o^2\big)=a,$$ and the graph of $R_I^i(a,\cdot)$
yields a parametrization of $\L(a)\cap \Sigma_I^i$.

To ease the notation, let us write $X:\bT^N\to\bR$, for the integrand 
of the integral that appeared in the definition of $G$ in \eqref{eq5.10}.
We may use a partition of unity 
\[
\{\var_I^i:\ i=1,\dots,\a(I),\ I\in\cI\},
\]
associated with the covering 
$\big\{\Sigma_I^i:\ i=1,\dots,\a(I),\ I\in\cI\big\}$, to write
\begin{align*}
G(a)=&\int_{\L(a)}X(\o)\ \s_{N-n}({\rm d}\o)\\
=&\sum_{I\in\cI}\sum_{i=1}^{\a(I)}
\int_{\L(a)\cap \Sigma^i_I}\big(\var_I^iX\big)(\o)\ \s_{N-n}({\rm d}\o)\\
=&\sum_{I\in\cI}\sum_{i=1}^{\a(I)}\int_{U^i_I}\big(\var_I^iX\big)\big(R_I^i(a,\o^2),\o^2\big)\ J_I^i(a,\o^2)\ {\rm d}\o^2,
\end{align*}
where $J^i_I$ is the corresponding Jacobian factor:
$$
J_I^i(a,\o^2)=\det \big(\rm I_{N-n}+E^i_I(a,\o_2\big))\big)^{1/2},
$$
where $\rm I_{N-n}$ is the identity matrix of $\bR^{N-n}$,
and $E^i_I\in \bR^{(N-n)\x(N-n)}$ is a matrix with $(r,s)$ entry
\[
\big(E^i_I\big)_{rs}=\frac{\p R_I^i}{\p \o_2^r}\cdot
 \frac{\p R_I^i}{\p \o_2^s}.
\]
From this representation,
 it is not hard to deduce the continuity of the map $a\mapsto G(a)$.

 Finally observe that when $a=0$, and
${\rm D}F(\bar \o)$ is invertible, then $\L(0)$ contains an $N-n$ dimensional 
surface (namely the graph of $R^i_I(0,\cdot)$). This implies that the right-hand side of \eqref{eq5.5} is not zero.
\end{proof}

Formula \eqref{eq5.5} offers an explicit expression for the fix points density
when $\cF(\widehat\Phi)=\widehat\Phi-{\rm id}$ is quasiperiodic. We still need to show that indeed a quasiperiodic $\cF(\Phi)$ yields a quaiperiodic $\cF(\widehat\Phi)$.
For this we need a refinement of Proposition~\ref{pro3.1}(i).

\ms
\bp\la{pro5.1} Assume that  $x\mapsto \Phi(x,\o)$ is a symplectic twist diffeomorphism such that $\Phi(x,\o)=x+K(\Theta_{Ax}\o)$, for a continuous function $K:\bT^N\to\bR^{2d}$. 
Then $\cF(\widehat\Phi)$ is quasiperiodic.
\ep

\begin{proof} Recall $\th_{(q,p)}=\eta_p\circ\tau_q=\tau_q\circ\eta_p$.
We also write
\[
\Phi(q,p,\o)=\big(q+\a(\tau_q\eta_p\o),p+\b(\tau_q\eta_p\o)\big).
\]
The twist condition means that the map
\[
q\mapsto \g(q,\o):=q+\a(\tau_q\o),
\]
is a diffeomorphism. If we write $\g^{-1}(Q,\o)$ for its inverse, and set
\be\la{eq5.16}
\hat\a(\o)=\g^{-1}(0,\o),
\ee
then we can then write
\[
\hat\a(\o)+\a\big(\tau_{\hat\a(\o)}\ \o\big)=0.
\]
From
\[
q+\a\big(\tau_q\o\big)=Q\ \ \ \Leftrightarrow  \ \ \ 
q-Q+\a\big(\tau_{q-Q}\tau_Q\o\big)=0,
\]
we deduce 
\[
\g^{-1}(Q,\o)=q=Q+\hat\a\big(\tau_Q\o\big).
\]
From this and the definition of $(\hat  q,\hat P)$ we learn
\begin{align*}
\hat q(Q,p)&=Q+\hat\a\big(\tau_Q\eta_p\o\big)
=Q+\hat\a\big(\th_{\hat x}\o\big),\\
\hat P(Q,p)&=P\big(\hat q(Q,p),p\big)=p+\b\big(\tau_Q\eta_p\tau_{\hat\a(\th_{\hat x}\o)}\o\big)\\
&=p+\hat\b\big(\th_{\hat x}\ \o\big),
\end{align*}
where $\hat x=(Q,p)$, and 
\be\la{eq5.17}
\hat\b(\o)=\b\big(\tau_{\hat\a(\o)}\o\big).
\ee
In summary, 
\be\la{eq5.18}
\cF(\widehat\Phi)(\hat x,\o)=\g(\th_{\hat x}\o), \ \ \ {\text{ where }}\ \ \ 
\g=(\hat\a,\hat\b),
\ee
with $\hat\a$ and $\hat\b$ as in \eqref{eq5.16} and \eqref{eq5.17}.

We now apply our general formula \eqref{eq5.18} to the case of a quasiperiodic
$\cF(\Phi)$. In this case, $\o\in\bT^N$, 
$\a,\b:\bT^N\to\bR^d$ are two continuous functions, and we have
 two $N\x d$ matrices $A^1$ and $A^2$, such that $A=[A^1,A^2]$, and
\begin{align*}
\tau_q\o&=\Theta_{A^1 q}\ \o=\o+A^1q\ \mod 1,\\
\eta_p\o&=\Theta_{A^2 p}\ \o=\o+A^2 p\ \mod 1.
\end{align*}
Analogously $\g=(\hat\a,\hat\b):\bT^N\to\bR^{2d}$ is a continuous
function such that \eqref{eq5.18} holds. This certainly implies the quasiperiodicity
of $\cF(\widehat\Phi)$.
\end{proof}

\subsection{Proof of Theorem~\ref{MGT}}
We are now ready to offer a more precise statement of Theorem~\ref{MGT}
and give a proof.

\bth\la{th5.2} Assume that  $x\mapsto \Phi(x,\o)$ is a ${\rm C}^2$ symplectic twist diffeomorphism such that $$\Phi(x,\o)=x+K(\Theta_{Ax}\o),$$ for a ${\rm C}^1$ function $K:\bT^N\to\bR^{2d}$. Let $\cQ$ denotes the Lebesgue measure
on $\bT^N$, and assume that $K\in \widehat{\mathfrak{H}}^{-1}(\cQ)$,
and 
\be\la{eq5.19}
\int_{\bT^N}K\ {\rm d}\cQ=0.
\ee
Then the set
\[
\big\{x\in\bR^{2d}:\ \Phi(x,\o)=x\big\},
\]
is of positive (possibly infinite) density, $\cQ$-almost surely.
\et

\begin{proof}
Our regularity assumption $K\in \widehat{ \mathfrak{H}}^{-1}(\cQ)\cap {\rm C}^1$, and
\eqref{eq5.19} allow us to apply Theorem~\ref{th3.3} to deduce the existence 
of a stationary generating function $w(x,\o)=\hat w(\th_x\o)$. By 
Proposition~\ref{pro5.1}, The map 
\[
\cF(\widehat \Phi)=\widehat \nabla w,
\]
is quasiperiodic. As we illustrated in Example~\ref{Example3.6}, the quasiperiodicity
of $\widehat \nabla w$ implies the quasiperiodicity of $w(x,\o)$. Proposition~\ref{pro4.1} guarantees the $\cQ$-almost sure existence of a density \eqref{eq4.2} for the set $Z(\o)$. When the right-hand side of \eqref{eq4.2} is infinite, there is nothing to prove. When the right-hand side of \eqref{eq4.2} is finite, we apply Theorem~\ref{th5.1} to find an explicit expression given by \eqref{eq5.5}
for the density. By choosing $U$ to be the set of all symmetric matrices (or even 
the set of positive or negative matrices), we can guarantee the positivity of the density of the set $Z(\o),$ $\cQ$-almost surely.
\end{proof}

\ms
\begin{remark} We refer to Example~\ref{Example3.6} for sufficient conditions that
would guarantee $K\in  \widehat{\mathfrak{H}}^{-1}(\cQ)$.
\end{remark}

\section{Stationary Hamiltonian ODEs}\la{ode}

In this section we study the time one map $\phi^H$ for a Hamiltonian function 
that is selected randomly according to a $\th$-invariant probability measure
$\bP$ on $\cH$. As it is well-known, there is a one-to-one correspondence between
$1$-periodic orbits of the Hamiltonian vector field $X_H(x,t)=J\nabla H(x,t) $
and the fixed points of $\phi^H$. The map $H\mapsto\phi^H$
pushes forward $\bP$ to a probability measure $\cP$ on $\cS$. To count the fixed points of $\phi^H$, we wish to apply Theorem~\ref{th3.3}. For this, we need 
to make some regularity assumptions on $H$, and verify
the applicability of Theorem~\ref{th3.3}. Let us first make a useful definition concerning the regularity of Hamiltonian functions.

\begin{definition}
Let us write $\cC^2(\ell)$ for the set of continuous
 maps $$H:\bR^n\x \bR\to\bR$$ such that $H$ is twice differentiable in $x$,
and 
\be\la{eq6.1}
\|{\nabla}_x H\|_{{\rm C}^0},\ \|{\rm D}_x^2 H\|_{{\rm C}^0}\leq \ell.
\ee
To ease the notation, we will write $\nabla$ and $\rm D$ for $\nabla_x$ 
and $\rm D_x$, respectively.
\end{definition}

\ms
In the next Proposition, we verify various properties of $\phi^H$ in terms
of the properties of $H$. This will prepare us to apply Theorem~\ref{th3.3}
to $\phi^H$, where $H$ is selected according to the $\th$-invariant measure $\bP$.

\bp\la{pro6.1} 
The following statements hold:
\begin{itemize}
\item[(i)] We have the following equalities $$\phi^{\th_a H}=\th_{-a}\circ \phi^H\circ \th_a= \th'_a\phi^H.$$ In particular, if
$\cG:\cH\to\cS$ is defined by $$\cG(H)=\phi^H_1=\phi^H,$$ then 
$\cG$ pushes forward any $\th$-invariant ergodic measure $\bP$
on $\cH$, to a $\th'$-invariant ergodic measure $\cP$
on $\cS$.

\ms
\item[(ii)] Let $\bP$ be a $\th$-invariant probability measure such such that 
\[
\int_{\cH} \|H\|_{{\rm C}^1}^{2d+1}\ \bP({\rm d}H)<\i.
\]
Then, 
$$
\int_{\cH}  \cG(H)(0)\ \bP({\rm d}H)=0.
$$

\ms
\item[(iii)] For $H\in \cC^2(\ell)$, we have $\cF(\phi^H)\leq \ell$, and 
\be\la{eq6.2}
\big\|{\rm D}\phi^{H}-I\|_{{\rm C}^0}\leq ( {\rm e}^{\ell}-1).
\ee
In particular, $\phi^H$ is a twist map if ${\rm e}^\ell<2$.

\ms
\item[(iv)] For $H,H'\in \cC^2(\ell)$, we have
\be\la{eq6.3}
\|\phi^{H'}-\phi^{H}\|_{{\rm C}^0}\leq {\rm e}^{\ell}\|\nabla H'-\nabla H\|_{{\rm C}^0}.
\ee

\ms
\item[(v)] Assume that $\bP$ is concentrated on $\cC^2(\ell)$ for some 
$\ell>0$. Assume 
\be\la{eq6.4}
\int_{\cH}\int_{\bR^{2d}}\int_0^1\int_0^1
|\nabla H(x,t)\cdot \nabla H(0,s)|\ |L(x)|\ {\rm d}x \, {\rm d}t \, {\rm d}s\ 
\bP({\rm d}H)<\i.
\ee
Then the map $H\mapsto \phi^H(0)$ is in $\widehat{\mathfrak H}^{-1}(\bP)$:
\be\la{eq6.5}
\int_{\cH}\int_{\bR^{2d}}
\big|\big(\phi^H(x)-x\big)\cdot \phi^H(0)|\ |L(x)|\ {\rm d}x\ \bP({\rm d}H)<\i.
\ee

\ms
\item[(vi)] If $H\in\cC^2(\ell)$ and $\nabla H$ is almost periodic, then 
$\cF\big(\phi^H\big)$ is almost periodic.

\ms
\item[(vii)]  Assume 
that $H\in\cC^2(\ell)$ and $ H$ is quasiperiodic i.e.
 we can find an integer $N\geq n$,
 a $1$-periodic function $$K:\bR^N\to\bR,$$
 a matrix $A\in\bR^{N\x n} $,  and $\o\in\bR^N$,
such that $$H(x,t)=H(x,t,\o)=K(\Theta_{Ax}\o,t).$$ (Here $\Theta_a$ denotes the translation
of $\bR^N$.)
Assume that the null set of $A$ is trivial:
\[
Ax=0\ \ \ \implies\ \ \  x=0.
\]
Then 
$\cF\big(\phi^H\big)$ is quasiperiodic.
\end{itemize}
\ep

\begin{proof} 
{(i)} This is an immediate consequence of the fact that 
if $y(\cdot)$ is an orbit of $X_{\th_a H}$, then $$x(\cdot)=\th_{a}
y(\cdot)=y(\cdot)+a$$ is an orbit of $X_{H}$.

\ms\noi
{(ii)}  Let us write ${\rm B}_\ell={\rm B}_\ell(0)$ for the ball of radius $\ell$
that is centered at the origin.
Since
\[
\phi^H_1(x)-x=\int_0^1J\nabla H\big(\phi^H_t(x),t\big)\ {\rm d}t,
\]
 we have 
\begin{eqnarray}\la{eq6.6}
\int_{{\rm B}_\ell}\big(\phi^H_1(x)-x\big)\ {\rm d}x &=& 
\int_0^1J\int_{{\rm B}_\ell}
\nabla H\big(\phi^H_t(x),t\big)\ {\rm d}x\  {\rm d}t  \\
&=& \int_0^1J\int_{\phi_t^H({\rm B}_\ell)}\nabla H(x,t)\ {\rm d}x\  {\rm d}t. \nonumber 
\end{eqnarray}
Note that since  
\begin{eqnarray}
\big|\phi_t^H(x)-x\big| &\leq&  t\sup|\nabla H| \nonumber \\
&\leq& \|H\|_{{\rm C}^1} \nonumber \\
&=:& c_0, \nonumber
\end{eqnarray}
for $t\in[0,1]$, we have
$$
{\rm B}_{\ell-c_0}\subset \phi_t^H({\rm B}_\ell)\subset {\rm B}_{\ell+c_0}.
$$
From this and \eqref{eq6.6} we learn
\begin{eqnarray}
&& \left|\int_{{\rm B}_\ell}\big(\phi^H_1(x)-x\big )\ {\rm d}x\right| \nonumber \\
&\leq& \left|\int_0^1J\int_{{\rm B}_\ell}\nabla H(x,t)\ {\rm d}x\  {\rm d}t\right|
+\int_0^1\int_{{\rm B}_{\ell+c_0}\setminus {\rm B}_{\ell-c_0}}|\nabla H(x,t)| \
{\rm d}x \ {\rm d}t \nonumber \\
&\leq& \left|\int_0^1J\int_{\p {\rm B}_\ell} H(x,t)\nu(x)\ \s({\rm d}x)\  {\rm d}t
\right|+c_0\big |{\rm B}_{\ell+c_0}\setminus {\rm B}_{\ell-c_0}\big|  \nonumber \\
&\leq& c_0\s\big(\p {\rm B}_\ell\big)+ c_0\int_{\ell-c_0}^{\ell+c_0}\s(\p {\rm B}_r)\ {\rm d}r \nonumber \\
&\leq& c_0\s\big(\p {\rm B}_1\big)\big(\ell^{2d-1}+
2c_0(\ell+c_0)^{2d-1}\big) \nonumber \\
&\leq& 
c_1\ell^{2d-1}(1+c_0)^{2d+1}, \nonumber
\end{eqnarray}
where $$\nu(x)=\frac{x}{|x|}$$ is the outward unit normal at $x\in\p {\rm B}_\ell$,  
$\s({\rm d}x)$ denotes the $2d-1$-surface
 measure on $\p {\rm B}_\ell$, and 
$c_1$ is a constant that depends on $d$ only.
 Hence, by stationarity of $\bP$, 
\begin{align*} 
\left|\int_{\cH} \phi^H_1(0)\ \bP({\rm d}H)\right|&=
\left|\int_{\cH} \left[|{\rm B}_\ell|^{-1}\int_{{\rm B}_\ell}\big(\phi^H_1(x)-x\big)\ {\rm d}x\right]
\ \bP({\rm d}H)\right|\\
&\leq c_2 \ell^{-1} \int_{\cH}\big(1+\|H\|_{{\rm C}^1})\big)^{2d+1} \ 
\bP({\rm d}H),
\end{align*}
for a constant $c_2$. We now send $\ell\to\i$ to complete the proof.

\ms\noi
(iii) Evidently,
\[
\big|\Phi^H(x)-x\big|\leq \left|\int_0^1 J\nabla H(\phi^H_s(x),s)\ {\rm d}s\right|\leq \ell.
\]
On the other hand, if 
\[
V(x,t)=J {\rm D}^2 H\big(\phi^H_t(x),t\big),\ \ \ \ A(x,t)={\rm D}\phi_t^H(x),
\]
then
\[
A_t(x,t)=V(x,t)A(x,t),
\]
which leads to the identity
\[
A(x,t)=I+\sum_{k=1}^\i\int_{\D_n(t)}V(x,t_n)\dots V(x,t_1)\ 
{\rm d}t_1\dots {\rm d}t_n,
\]
where
\[
\D_n(t)=\big\{(t_1,\dots,t_n):\ 0\leq t_1\leq \dots\leq t_n\leq t\big\}.
\]
From this we deduce \eqref{eq6.2}
because $|V|\leq \ell$ by \eqref{eq6.1}.

\ms\noi
(iv) Note that if $H\in \cL^2(\ell)$, then \eqref{eq6.1} implies that the Lipschitz constant of the vector field $J\nabla H$ is at most $\ell$.
Using this, we can write,
\begin{align*}
\frac {\rm d}{{\rm d}t}\left[{\rm e}^{-\ell t}\big|\phi_t^{H'}(x)-\phi_t^H(x)\big|\right]\leq&
{\rm e}^{-\ell t}\left|J\nabla H'(\phi_t^{H'}(x),t)-J\nabla H(\phi_t^H(x),t\big)\right|\\
&-\ell {\rm e}^{-\ell t}\big|\phi_t^{H'}(x)-\phi_t^H(x)\big|\\
\leq&
{\rm e}^{-\ell t}\left|J\nabla H(\phi_t^{H'}(x),t)-J\nabla H(\phi_t^H(x),t\big)\right|\\
&+{\rm e}^{-\ell t}\|\nabla H'-\nabla H\|_{{\rm C}^0}
-\ell {\rm e}^{-\ell t}\big|\phi_t^{H'}(x)-\phi_t^H(x)\big|\\
\leq &{\rm e}^{-\ell t}\|\nabla H'-\nabla H\|_{{\rm C}^0}.
\end{align*}
Integrating both sides with respect to $t$ yields
\[
{\rm e}^{-\ell }\big|\phi^{H'}(x)-\phi^H(x)\big|
\leq \|\nabla H'-\nabla H\|_{{\rm C}^0},
\]
which in  \eqref{eq6.3}.

\ms\noi
(v)
By stationarity, the left-hand side of \eqref{eq6.5} equals to
\begin{align*}
&\frac 1{|{\rm B}_1(0)|}\int_{|a|\le 1}\int_{\cH}\int_{\bR^{2d}}
\left|\big(\phi^H(x)-x\big)\cdot \big(\phi^H(a)-a\big)\right|\ |L(x-a)|\ {\rm d}x {\rm d}a\ \bP({\rm d}H)\\
&=\frac 1{|{\rm B}_1(0)|}\int_{|a|\le 1}\int_{\cH}\int_{\bR^{2d}}
\left|\left(\int_0^1X\big(\phi_t^H(x),t\big)\ {\rm d}t\cdot \int_0^1 X
(\phi_s^H(a),s)\ ds\right)\right|\ |L(x-a)|\ {\rm d}x \ {\rm d}a\ \bP(\rm dH)\\
&\le \frac 1{|{\rm B}_1(0)|}\int_{\cH}\int_0^1\int_0^1\int_{|a|\le 1}\int_{\bR^{2d}}
\left|\nabla H\big(\phi_t^H(x),t\big)\cdot \nabla H
(\phi_s^H(a),s)\right|\ |L(x-a)|\ {\rm d}x \  {\rm d}a\ {\rm d} t \ {\rm d}s\ \bP(\rm dH),
\end{align*}
where $X(x,t)=J\nabla H(x,t)$. We now make a (volume preserving) change of
variable 
\[
(y,z)=\big(\phi_t^H(x),\phi_s^H(a)\big),
\]
for the ${\rm d}x\rm da$ integration to rewrite the last expressions as
\[
\frac 1{|{\rm B}_1(0)|}\int_{\cH}\int_0^1\int_0^1\int_{\phi_s^H({\rm B}_1(0))}\int_{\bR^{2d}}
\left|\nabla H(y,t)\cdot \nabla H
(z,s)\right|\ \big|L\big(\psi_t(y)-\psi_s(z))\big|\ 
{\rm d}y \ {\rm d}z\ {\rm d}t \ {\rm d}s\ \bP(\rm dH),
\]
where $\psi_t$ is the inverse of $\phi_t^H$. This expression is bounded above
by 
\[
\frac 1{|{\rm B}_1(0)|}\int_{\cH}\int_0^1\int_0^1\int_{{\rm B}_{\ell+1}(0)}\int_{\bR^{2d}}
\left|\nabla H(y,t)\cdot \nabla H
(z,s)\right|\ \big|L\big(\psi_t(y)-\psi_s(z))\big|\ {\rm d} y \ {\rm d} z \ {\rm d}t \ {\rm d}s\ \bP(\rm dH),
\]
because 
\[
\phi_s^H({\rm B}_1(0))\subset {\rm B}_{\ell+1}(0),
\]
$\bP$-almost surely by \eqref{eq6.1}.
Observe that $\psi_t$ is the flow of the Hamiltonian ODE associated with $J\nabla H$
with time reversed. So using \eqref{eq6.1},
\[
|\psi_t(y)-y|\ ,\ |\psi_s(z)-z|\le \ell,
\]
$\bP$-almost surely, for $y,z\in\bR^{2d}$, and $s,t\in[0,1]$. As a result,
\be\la{eq6.7}
|y-z|\geq 4\ell\ \ \  \implies \ \ \  \big|\psi_t(y)-\psi_s(z)\big|\geq |y-z|-2\ell
\geq \frac 12|x-y|.
\ee
We now assume that $d>1$ so that $L(x)$ is given by a constant multiple
of $|x|^{2-2d}$. From of this, and \eqref{eq6.7}, we learn that
the left-hand side of \eqref{eq6.5} is bounded above by
\[
\L_1+\L_2,
\]
where
\begin{align*}
\L_1=&\frac 1{|{\rm B}_1(0)|}\int_{\cH}\int_0^1\int_0^1\iint_{E_{\ell}}
\left|\nabla H(y,t)\cdot \nabla H
(z,s)\right|\ \big|L\big(\psi_t(y)-\psi_s(z))\big|\  {\rm d} y \ {\rm d} z \ {\rm d}t \ {\rm d}s\ \bP(\rm dH),\\
\L_2=&c_0\ \frac 1{|{\rm B}_1(0)|}\int_{\cH}\int_0^1\int_0^1\int_{{\rm B}_{\ell+1}(0)}\int_{\bR^{2d}}
\left|\nabla H(y,t)\cdot \nabla H
(z,s)\right|\ \big|L(y-z)|\  {\rm d} y \ {\rm d} z \ {\rm d}t \ {\rm d}s\ \bP(\rm dH),
\end{align*}
for a constant $c_0$, and
\[
E_\ell=\big\{(y,z)\in\bR^{2d}:\ |y|\le \ell+1,\ |y-z|\le 2\ell\big\}.
\]
By stationarity of $\bP$, 
\begin{align*}
\L_2=&c_0\ \frac 1{|{\rm B}_1(0)|}\int_{\cH}\int_0^1\int_0^1\int_{{\rm B}_{\ell+1}(0)}\int_{\bR^{2d}}
\left|\nabla H(y-z,t)\cdot \nabla H
(0,s)\right|\ \big|L(y-z)|\  {\rm d} y \ {\rm d} z \ {\rm d}t \ {\rm d}s\ \bP(\rm dH)\\
=&c_0\ \frac {|{\rm B}_{\ell+1}(0)|}{|{\rm B}_1(0)|}\int_{\cH}\int_0^1\int_0^1\int_{\bR^{2d}}
\left|\nabla H(x,t)\cdot \nabla H
(0,s)\right|\ \big|L(x)|\ {\rm d}x\ {\rm d} t \ {\rm d}s\ \bP(\rm dH),
\end{align*}
and this is finite by our assumption \eqref{eq6.4}. It remains to show $\L_1<\i$.
Indeed from our assumption, \eqref{eq6.1} holds $\bP$-almost surely, which yields
the bound
\[
\L_1\le \frac {\ell^2}{|{\rm B}_1(0)|}\int_{\cH}\iint_{E_{\ell}}
 \big|L\big(\psi_t(y)-\psi_s(z)\big)\big|\  {\rm d} y \ {\rm d} z\  \bP(\rm dH).
\]
We make the change of variable 
\[
(x,a)=\Psi(y,z):=\big(\psi_t(y),\psi_s(z)\big),
\]
to rewrite the integral as
\[
\int_{\cH}\iint_{\Psi(E_{\ell})}
 |L(x-a)|\ {\rm d}x \ {\rm d}a\  \bP(\rm dH).
\]
Using \eqref{eq6.1}, we have
\[
\Psi(E_{\ell})\subset {\rm B}_{2\ell+1}(0)\x {\rm B}_{4\ell+1}(0).
\]
From this and local integrability of $L(x)$ we deduce that $\L_1<\i$.

The case $d=1$ can be treated in a similar fashion.

\ms\noi
(vi) Let us write
\begin{eqnarray}
O(H)&:=&\{\th_a H:\ a\in\bR^n\big\}, \nonumber \\
\widehat O(H)&:=&\{\th'_a\phi^H:\ a\in\bR^n\big\}. \nonumber
\end{eqnarray}
By part (i), we know 
\be\la{eq6.8}
\cG\big(O(H)\big)=\widehat O(H).
\ee
If $H\in \cC^2(\ell)$ and $\nabla H$ is almost periodic, then $$O(H)\subset \cC^2(\ell),$$
and $O(\nabla H)$ is precompact with respect to ${\rm C}^0$-topology.
From this and \eqref{eq6.8} we deduce the precompactness of $\widehat O(H)$.
Since 
\[
\th'_a\phi^H-\th'_b\phi^H=\th_a\cF(\phi^H)-\th_b\cF(\phi^H),
\]
we deduce the precompactness of the set $\big\{\th_a\cF(\phi^H)\big\}$.
As a consequence, the map 
$\cF\big(\phi^H\big)$ is almost periodic.

\ms\noi
(vii)  Set $\z_t(\o)=\phi_t^{H(\cdot,\o)}(0)$.
We claim that $\z$ is periodic.
Observe that if $x(t)$ solves the ODE
\[
\dot x(t)=J\nabla H(x(t),t)=JA^*\nabla K(\o+Ax(t),t),\ \ \ \ x(0)=0,
\]
then $\o(t):=\o+Ax(t)$ satisfies
\[
\dot \o(t)=AJA^*\nabla K(\o(t),t).
\]
From this we learn that if $\psi_t$ is the flow of the vector field
$$\hat X:=AJA^*\nabla K,$$  then
\[
\psi_t(\o)=\o+A\ \phi_t^{H(\cdot,\o)}(0)=\o+A\z_t(\o).
\]
 Since $\hat K$ is periodic, we learn that $\cF(\psi_t)=\psi_t-{\rm id}$ is periodic
by part (i). Hence $A\z_t(\o)$ is periodic. 
Since $A$ has a trivial null space, we deduce that $\z_t$ is periodic.
On the other hand,
\begin{align*}
\phi^{H(\cdot,\o)}(x)-x=&\big(\th'_x \phi^{H(\cdot,\o)}\big)(0)=
\big( \phi^{\th_xH(\cdot,\o)}\big)(0)=
\big( \phi^{H(\cdot,\th_x\o)}\big)(0)\\
=&\z_1(\th_x\o)=\z_1(\o+Ax).
\end{align*}
From this, and the periodicity of $\z_1$, 
we deduce the quasiperiodicity of the left-hand side.
\end{proof} 

\subsection{Proof of Theorem~\ref{th1.3}}
We are now ready to offer a more precise statement of Theorem~\ref{th1.3}
and give a proof.

\bth\la{th6.2} Assume that  $H(x,t,\o)=K(\Theta_{Ax}\o,t)$, for a ${\rm C}^1$ function $K:\bT^{N}\x\bT\to\bR$. Let $\bP$ denotes the Lebesgue measure
on $\bT^N$, and assume $K\in\cC^2(\ell)$, for some $\ell\in(0,\log 2)$,
and that \eqref{eq6.4} holds.
Then the set
\[
\big\{x\in\bR^{2d}:\ \phi^{H(\cdot,\o)}(x)=x\big\},
\]
is of positive (possibly infinite) density, $\bP$-almost surely.
\et

\begin{proof}
It suffices to show that the conditions of Theorem~\ref{th5.1} hold true
for 
\[
\Phi(x,\o)=\phi^{H(\cdot,\o)}(x).
\]
These conditions have been verified in Propositions~\ref{pro6.1},
parts (i)--(iii), (v), (vii).
\end{proof}

\ms
\begin{remark} As we discussed in Example~\ref{Example3.6},
a Diophantine-type condition on $A$, and the existence of certain number of derivatives of the function $K$ would guarantee the validity of \eqref{eq6.4} for the corresponding Hamiltonian function $H$ in the quasiperodic setting of Theorem~\ref{th6.2}.
\end{remark}

\section{Appendix. The $2$ dimensional case: Random Poincare\--Birkhoff theorem} \label{sec2}
 
 To put in context the stochastic 
 Conley\--Zehnder theory we develop in this paper, we very briefly review the simpler case of stochastic symplectic maps in dimension $2$.
 In order to describe our results, 
let us write, following~\cite{PR}, $\cT$ for the space of area
preserving twist maps.  Let $\overline \cT$ be the space of 
maps $$\bar F:(\cA= \mathbb{R}\times [-1,\,1])  \to\cA$$ such that if 
$$
\ell(\bar F)(q,p):=(q,0)+\bar F(q,p),
$$
then $\ell(\bar F)\in \cT$. Consider the operator $\ell:\overline \cT\to\cT$ which send $\bar F$ to 
$F=\ell(\bar F)$. 
So we have a family of shifts $$\big\{\tau_a:\overline\cT\to\overline \cT: \ a\in\bR\big\},$$ 
defined by
$
\tau_a\bar F(q,p)=\bar F(q+a,p).
$
For any $ F\in\cT$ we write 
\[
{\rm Fix}( F)=\big\{x\in\cA:\ F(x)=x\big\}.
\]
Also, with a slight abuse of notation we write $\tau_a$ instead of
$$
\tau_aA=\{x:\ x+a\in A\}.
$$
 We then have the trivial commutative relationship
\begin{equation}\label{shift}
\tau_a {\rm Fix}(\ell(\bar F))={\rm Fix} (\ell(\tau_a \bar F)).
\end{equation}
Furthermore,  we adopt the following notations and terminologies: we denote by
 $\overline{\cM\cT_+}$ the space of $$\bar F=(\bar Q,\bar P)\in\overline\cT$$ 
with the property that for every $q$ we have that the function  $f\colon [-1,\,1]\to \mathbb{R}$ defined by $f(p):=\bar Q(q,\, p)$ is 
 increasing; we denote by $\cM\cT_+$ the space of $\ell(\bar F)$, with $\bar F\in \overline{\cM\cT_+}$; we denote by $\cM\cT_-$ the space of  $F$ such that $F^{-1}\in\cM\cT_+$; 
 the elements of $\cM\cT_+$ are the positive monotone twist maps; the elements  of $\cM\cT_-$ are the negative monotone twist maps; 
a fixed point $x=(q,\,p)$ of
$F(\cdot,\,\cdot) \colon \mathcal{S} \to
\mathcal{S}$ is of  $+$
(respectively $-$) type if the eigenvalues of ${\rm d}F(q,\,p)$ are positive 
(respectively negative); finally we write
\[
{\rm Fix}_{\pm}(F)=\big\{x \in {\rm Fix}(F):\ x {\text{ is of }}\pm{\text{ type}}\big\}.
\]

For any 
 $F=\ell(\bar F)\in\cM\cT_+$ there is  a scalar\--valued function $$\cG(q,Q)=\cG(q,Q;\bar F)$$ such that 
\begin{equation}\label{gen}
F(q,-\cG_q(q,Q))=(Q,\cG_Q(q,Q)).
\end{equation}
Due to the existence of boundary conditions for $$F(q,p)=(P(q,p),Q(q,p)),$$ 
we only need to define $\cG(q,Q)$ for $(q,Q)$ such that
$Q(q,-1)\leq Q\leq Q(q,+1)$. This means in particular that
$
\psi(q;\bar F):=\cG(q,q;\bar F)
$
is well defined. By~\eqref{shift}, $$\psi(\cdot;\tau_a\bar F)=\tau_a\psi(\cdot;\bar F).$$

In~\cite[Theorem~B]{PR} we saw that if $\bQ$ is a translation invariant ergodic probability measure
on $\big(\overline\cT,\cF\big)$ satisfying that $$\bQ\big(\overline{\cM\cT_+}\big)=1$$ then all the sets
${\rm Fix}_{\pm}\big(\ell(\bar F)\big)$ are nonempty with probability one with respect to $\bQ$. Moreover,
also with probability one, if the random pair 
\[
\left(\frac {\rm d}{{\rm d}q}\psi(q;\bar F),\frac {{\rm d}^2}{{\rm d}q^2}\psi(q;\bar F)\right),
\]
  has a probability density $\rho(a,b;\bar F)$ (which is independent of $q$ due to the translation invariance), 
then the sets ${\rm Fix}_{\pm}\big(\ell(\bar F)\big)$ have positive density $\l_{\pm}$ given by
\[
\l_{\pm}=\int\left[\int_{-\infty}^{\infty} b^{\pm}\rho(0,b;\bar F)\ {\rm d}b\right]\ \bQ({\rm d}\bar F).
\]

Now denote by $\O_0$ the space of functions $\o \colon \bR^2\to\bR$
such that $\o(q,a)>0$ for $a>0$, $\o(q,0)=0$,  and, 
$$
\eta(q;\o)=\inf\{a\ |\ \o(q,a)=2\}<\i,
$$
for every $q$.  Define 
\begin{eqnarray}
Q^-(q;\o)&=&\frac 12\int_0^{\eta(q;\o)} \o(q,a) \ {\rm d}a-\eta(q;\o); \nonumber\\
G(q,Q;\o)&=&\o(q,Q-q-Q^-(q;\o)). \nonumber
\end{eqnarray}
Denote by $\O_1$  the space of all $\o\in\O_0$ satisfying that $G_q(q,Q;\o)<0$ for every $(q,Q)$.
In \cite[Theorem~C]{PR} we proved that for each $\o\in\O_1$ there is a unique function 
$\bar F(\cdot,\cdot;\o)$ such that if 
$$F(\cdot,\cdot;\o)=\ell\big(\bar F(\cdot,\cdot;\o)\big)$$ then  (\ref{gen}) holds for
$\cG(q,Q)=\cG(q,Q;\o)$, that is given by
\[
\cG(q,Q;\o)=\int_{q+Q^-(q;\o)}^{Q}\o(q,a) \ {\rm d}a-(Q-q).
\]
Moreover, we proved that if $$\tau_a\o(q,v)=\o(q+a,v)$$
then we have that $$\bar F(\cdot,\cdot;\tau_a\o)=\tau_a\bar F(\cdot,\cdot;\o).$$

To continue the discussion of results on the $2$\--dimensional case let $\O_2$ denote the set ${\rm C}^2$ Hamiltonian functions $\o(q,p,t)$ with the property that they have uniformly bounded second derivatives 
and such that $\pm \o_p(q,\pm 1,t)>0$ and $\o_q(q,\pm 1,t)=0$. If $\o\in\O_2$ let
$$\tau_a\o(q,p,t)=\o(q+a,p,t)$$ as before, and also let $$\phi_t^\o(q,p)$$ be the flow of the corresponding Hamiltonian system $$\dot q=\o_p(q,p,t), \dot p=-\o_q(q,p,t).$$
One can show that if 
$$F^t(q,p;\o)=\phi_t^\o(q,p)$$ and $$\bar F^t(q,p;\o)=\phi_t^\o(q,p)-(q,0)$$ then $$F^t(\cdot,\cdot;\o)\in\cT$$ and 
$$
\bar F^t\big(q,p;\tau_a\o\big)=\tau_a\bar F^t\big(q,p;\o\big).
$$
With this in mind, we proved in~\cite[Theorem~D]{PR} that 
if $\bP$ is a $\tau$-invariant ergodic probability measure $\bP$ on $\O_2$,  then for every $t \geq 0$ we have that
$$\bP(\# {\rm Fix}(F^t(\cdot,\,\cdot,\,;\omega))=\infty)=1.$$

Denote by $\cC([0,1];\overline\cT)$ the space of 
${\rm C}^1$ maps  $\g:[0,1]\to \overline\cT$ for which $\ell(\g(0))$ is the identity.
The operator $\tau$ on $\overline\cT$ induces a new  operator (denoted in the same way)
 on $\cC([0,1];\overline\cT)$ defined by $$(\tau_a\g)(t)=\tau_a(\g(t)).$$ 
 
We know that if  $\bP$ is a stationary ergodic measure on $\O_2$ then
\[
\o\mapsto \big( \bar F^t(q,p;\o)=\phi_t^\o(q,p)-(q,0):\ t\in[0,1]\big),
\] 
pushes forward $\bP$ onto a stationary ergodic probability measure $\cQ$ on $\cC([0,1];\overline\cT)$.
The converse also holds;  a stationary ergodic probability measure $\cQ$ on $\cC([0,1];\overline\cT)$ 
always comes from a unique a stationary ergodic measure $\bP$ on $\O_2$. 

Let $\bQ$ be  a stationary ergodic measure on $\overline\cT$. A natural question is whether we can find a stationary ergodic measure $\cQ$ on $\cC([0,1];\overline\cT)$
 such that $\bQ$ is the push forward of $\cQ$ under the time-1 map
$\pi_1:\cC([0,1];\overline\cT)\to \overline\cT$ (by time-1 map we mean $\pi_1\g:=\g(1)$).  In order to discuss this question
let $\cD$ denote the space of diffeomorphisms $F:\cA\to\cA$.  
We also denote by $\overline\cD$ the space of functions $\bar F$ with the property that $\ell(\bar F)\in\cD$. 

Let $\cQ$ be a stationary ergodic measure on $\cC([0,1];\overline\cD)$. In~\cite{PR} we called $\cQ$  regular if 
\[
\int \sup_{t\in[0,1]}\left[\left\|\dot\g(t)\right\|_{\i}+\left\|{\rm d}\g(t)\right\|_{\i}+
\left\|{\rm d}\g(t)^{-1}\right\|_{\i}\right]\ \cQ({\rm d}\g)<\i.
\]
Here $\|\cdot\|_\i$ denotes the ${\rm L}^\i$ norm and $\dot\g(t)$ and  ${\rm d}\g(t)$ denote the derivatives of $\g(t)$
with respect to $t$ and $x=(q,p)$, and 
$$\frac 12\int\left[\int_{-1}^{1} \det({\rm d} \g(t)(q,p))\,{\rm d}p\right]\ \cQ({\rm d}\g)=1;$$ 
for all $t\in[0,1]$. 

If we start with a stationary ergodic measure $\bQ$ on $\overline\cT$, in \cite[Theorem~E]{PR} we proved 
that if there exists  a regular stationary ergodic measure $\cQ$ on $\cC([0,1];\overline\cD)$,
 such that $\bQ$ is the push forward of $\cQ$ under the time-1 map
 $\pi_1\g:=\g(1)$, then there is another  stationary ergodic measure $\cQ'$ on $\cC([0,1];\overline\cT)$
with the property that $\bQ$ is the push forward of $\cQ$ under $\pi_1$.

Finally, our paper~\cite{PR} concluded by showing that 
if $\bP$ is a $\tau$-invariant ergodic probability measure $\bP$ on $\O_2$ and  $F=F^1$ is as in the result we  described earlier,
 then there exists~\cite[Theorem~E]{PR} a deterministic integer $N \geq 0$ and  area\--preserving random twists
$F_j$,  $0 \leq j \leq N,$
such that for $\bP$ almost all $\omega \in \Omega_2$, we have a decomposition:
$$
F(\cdot,\cdot;\omega)=F_N(\cdot,\cdot;\omega)\circ \ldots\circ F_2(\cdot,\cdot;\omega)\circ F_1(\cdot,\cdot;\omega)\circ F_0(\cdot,\cdot;\omega),
$$
where  the map $F_j$ is positive monotone if $j$ is an odd integer, $F_j$ is negative monotone if  $j$ is an even integer,   and 
$$\bar F_j(q,p;\tau_a\o)=\tau_a\bar F_j(q,p;\o)$$ for every $j$.

\bigskip

\bigskip

\textup{\,}
\\
{\bf \'Alvaro Pelayo} \\
alvpel01@ucm.es\\
Facultad de Ciencias Matem\'aticas \\
Universidad Complutense de Madrid\\
 28040 Madrid, Spain

\textup{\,}
\\
{\bf Fraydoun Rezakhanlou} \\
rezakhan@math.berkeley.edu\\
Department of Mathematics \\
University of California, Berkeley \\
Berkeley, CA 94720-3840 USA

\end{document}